\documentclass[a4paper,10pt]{article}
\usepackage[utf8]{inputenc}

\usepackage{comment,xcomment,amssymb,amsmath,color}
\usepackage{rotating,xypic,array}
\usepackage{latexsym,mathtools,amsthm}
\usepackage{subfig}
\usepackage{tikz,fancyvrb}
\usepackage{booktabs}  
\usetikzlibrary{arrows.meta,automata}
\usepackage{enumitem}
\usepackage{url} 
\usepackage{longtable}
\usepackage[section]{placeins}
\usepackage{multirow,rotating,graphicx}
\usepackage{pdflscape}

\newcommand{\rn}[2]{
    \tikz[remember picture,baseline=(#1.base)]\node [inner sep=0] (#1) {$#2$};%
}

\title{Diagram involutions and homogeneous Ricci-flat metrics}
\author{Diego Conti, Viviana del Barco, and Federico A. Rossi}

\makeatletter
\def\namedlabel#1#2{\begingroup
   \def\@currentlabel{#2}%
   \label{#1}\endgroup
}
\makeatother

\newtheorem{theorem}{Theorem}[section]
\newtheorem{lemma}[theorem]{Lemma}

\newtheorem{corollary}[theorem]{Corollary}
\newtheorem{proposition}[theorem]{Proposition}
\theoremstyle{definition}
\newtheorem{definition}[theorem]{Definition}
\newtheorem{example}[theorem]{Example}

\theoremstyle{remark}
\newtheorem{remark}[theorem]{Remark}

\newcommand{\abs}[1]{\left\vert#1\right\vert}
\newcommand{\R}{\mathbb{R}}
\newcommand{\lie}[1]{\mathfrak{#1}}     
\newcommand{\g}{\lie{g}}
\newcommand{\Z}{\mathbb{Z}}

\newcommand{\hook}{\lrcorner\,}

\newcommand{\Sp}{\mathrm{Sp}}

\newcommand{\SO}{\mathrm{SO}}

\newcommand{\Gtwo}{\mathrm{G}_2}

\newcommand{\so}{\mathfrak{so}}

\newcommand{\GL}{\mathrm{GL}}
\newcommand{\SL}{\mathrm{SL}}

\newcommand{\Span}[1]{\operatorname{Span}\left\{#1\right\}}

\DeclareMathOperator{\ric}{ric}  

\DeclareMathOperator{\ad}{ad}

\newcolumntype{C}{>{$}c<{$}}
\newcolumntype{L}{>{$}l<{$}}
\newcolumntype{R}{>{$}r<{$}}

\newcommand{\B}{\mathcal{B}}
\newcommand{\st}{\:\mid\:}          
\newcommand{\eps}{\varepsilon}   

\newcommand{\coord}{\mathrm{coord}}

\newcommand{\spg}{\mathfrak{sp}}

\newcommand{\mg}{\mathfrak{g}}
\newcommand{\mb}{\mathfrak{b}}

\newcommand{\mn}{\mathfrak{n}}
\newcommand{\mz}{\mathfrak{z}}
\newcommand{\mk}{\mathfrak{k}}

\newcommand{\mh}{\mathfrak{h}}
\newcommand{\ma}{\mathfrak{a}}
\newcommand{\mgg}{\mathfrak{g}}
\newcommand{\mpp}{\mathfrak{p}}

\newcommand{\bil}{\left\langle \;,\;\right\rangle}
\newcommand{\lela}{\left\langle}
\newcommand{\rira}{\right\rangle}

\newcommand{\noi}{\noindent}
\newcommand{\mcB}{\mathcal{B}}
\newcommand{\mcX}{\mathcal{X}}
\newcommand{\mcZ}{\mathcal{Z}}
\newcommand{\mnT}{\mn_\Theta}

\begin{document}

\maketitle

\begin{abstract}
We introduce a combinatorial method to construct indefinite Ricci-flat metrics on nice nilpotent Lie groups.

We prove that every nilpotent Lie group of dimension $\leq6$, every nice nilpotent Lie group of dimension $\leq7$ and every two-step nilpotent Lie group attached to a graph admits such a metric.  We construct infinite families of Ricci-flat nilmanifolds associated to parabolic nilradicals in the simple Lie groups $\SL(n)$, $\SO(p,q)$, $\Sp(n,\R)$. Most of these metrics are shown not to be flat.

\end{abstract}

\renewcommand{\thefootnote}{\fnsymbol{footnote}}
\renewcommand{\thefootnote}{\arabic{footnote}}

\section*{Introduction}
The construction of metrics satisfying the Einstein equation
\begin{equation}
 \label{eqn:einstein}
\ric= \lambda g\end{equation}
 is a classical problem in Riemannian geometry. Both explicit examples and general conditions for the existence of a solution on a given manifold are now known, mostly in the context of K\"ahler geometry and special holonomy, or more generally special geometries (see~\cite{Besse} and the references therein). Nonetheless,  a complete classification appears to be hopeless.

The homogeneous case is deserving of particular attention, in that the second order PDE~\eqref{eqn:einstein} is turned into a set of polynomial equations, though generally fairly complicated. Indeed, even in the homogeneous context, whilst both sufficient and necessary conditions on a homogeneous manifold are known for the existence of a Riemannian Einstein metric (see~\cite{WangZiller}), a classification has not yet been achieved. Among homogeneous Einstein Riemannian manifolds one finds irreducible symmetric spaces and more generally isotropy irreducible spaces, classified in~\cite{Wolf:TheGeometry}, as well as simple Lie groups (with a non-unique metric, see~\cite{Jensen:scalar}), and on the other hand Einstein solvmanifolds, whose structure is reduced to the study of particular metrics on nilpotent Lie groups, known as nilsolitons (\cite{Heber:noncompact,Lauret:Einstein_solvmanifolds}). A remarkable link between the two settings was established in~\cite{Tamaru} by Tamaru, who showed that under the natural embedding of a parabolic nilradical in the corresponding symmetric space, the ambient Einstein metric reduces to a nilsoliton metric. In fact, all known homogeneous Riemannian Einstein manifolds with negative curvature can be realized as solvmanifolds, motivating the so-called Aleksveesky conjecture.
Notice that the case of Ricci-flat homogeneous metrics, i.e. with Ricci tensor equal to zero, is essentially trivial in that by~\cite{AlekseevskiKimelFel} it only gives flat metrics.

The pseudoriemannian side of  the story is quite different. Whilst the typical constructions of Riemannian homogeneous Einstein manifolds have an indefinite counterpart (see~\cite{Kath:PseudoRiemannian,Derdzinski,DerdzinskiGal}), they are far from exhausting Einstein indefinite homogeneous metrics, due to  a much greater flexibility. In particular, it is well known that a homogeneous Ricci-flat indefinite metric is not necessarily flat.  Examples  arise in the ad-invariant context (which however gives very few examples, see e.g.~\cite{dBO12,Favre,Ov16}), low-dimensional solvable Lie groups (see~\cite{CalvarusoZaeim:lorentzian,CalvarusoZaeim:neutral} for a classification  in dimension four, or~\cite{ContiRossi:RicciFlat} for some sporadic examples in dimension $\leq 8$), step two nilpotent Lie groups (\cite{Boucetta,Guediri2003,GuediriBinAsfour}), and associated to metrics with special holonomy, or more generally special geometries (see~\cite{AlekseevskyMedoriTomassini:Homogeneous2009,
FinoLujan:TorsionFreeG22,FinoKath,
Freibert:calibrated,IvanovZamkovoy:parahermitian,Kath:Indefinite,
SchaferSchulteHengesbach}).

Even though these examples appears to suggest that many  homogeneous manifolds admit an invariant Ricci-flat metric, a systematic construction appears to be lacking at the time of writing.

This paper gives a contribution in this direction by introducing a systematic construction of Ricci-flat indefinite metrics on a large class of homogeneous spaces. Inspired by the Riemannian situation, which seems to suggest that an understanding of homogeneous Einstein metrics will entail a thourough study of the nilpotent case, we focus on the case of left-invariant metrics on  nilpotent Lie groups. Our construction applies in particular to many  nilpotent Lie groups satisfying the sufficient condition given in~\cite{Malcev} for the existence of a lattice, giving rise to compact Ricci-flat manifolds (in fact, we obtain infinitely many distinct diffeomorphism types, see Remark~\ref{remark:infinitelymanydiffeomorphismtypes}). Nonetheless, an analysis of their corresponding curvature tensor allows us to show that in most situations the left-invariant Ricci-flat metrics can be chosen to be nonflat.

Our construction applies to nice nilpotent Lie algebras. Nice Lie algebras were introduced in~\cite{LauretWill:Einstein} in the context of the construction of Einstein Riemannian metrics on solvmanifolds (more precisely, the attached nilsolitons), and they find application in other geometric problems involving the Ricci tensor (see e.g.~\cite{LauretDere:OnRicciNegative,LauretWill:diagonalization}); their structure is largely described by a kind of directed graph known as a nice diagram (\cite{ContiRossi:Construction}). Whilst our ultimate goal is of a purely geometric nature, it is this combinatorial nature of nice Lie algebras that enables us to produce Ricci-flat metrics through combinatorial tools.

Indeed, we consider involutions of the nice diagram that satisfy what we call the \emph{arrow-breaking} condition (see Definition~\ref{def:arrowbreaking}); loosely speaking, this means they are as far as possible from being automorphisms of the diagram. This should be contrasted with~\cite{ContiRossi:RicciFlat}, where diagram automorphisms were used to produce Ricci-flat metrics. The present method appears to be much more effective, partly because it requires no additional computation once the combinatorial condition is satisfied.

The other measure of the effectiveness of arrow-breaking involutions is the fact that it produces Ricci-flat metrics on all the nice nilpotent Lie algebras of dimension $\leq7$ except one (see Theorem~\ref{teo:nicedimless7}). This abundance of examples  suggests a natural question:

\begin{center}
\emph{does every nilpotent Lie algebra admit a Ricci-flat metric?}
\end{center}
We  answer this question in the positive for dimension $\leq6$, by producing explicit Ricci-flat metrics on the Lie algebras not covered by the arrow-breaking construction. In addition, we show that every nice nilpotent Lie algebra of dimension $\leq 7$ has a Ricci-flat metric, which can be chosen to be nonflat, except for abelian Lie algebras and two low-dimensional exceptions. We leave the above question open for the $34+ 5$ families of nilpotent Lie algebras of dimension 7 that do not admit a nice basis listed in~\cite{ContiRossi:Construction}, as well as for higher dimensions. It is worth noticing that the same question for flat metrics was answered negatively in \cite{AM03}, namely, there exist nilpotent Lie algebras not admitting any flat metric.

We consider two standard constructions to produce infinite families of nice nilpotent Lie algebras, namely the two-step nilpotent Lie algebras associated to a graph (see~\cite{DaniMainkar:AnosovAutCompactNil}), and the nilradicals of parabolic subalgebras of split simple Lie algebras (see~\cite{Ko1}). We  apply our construction to produce infinite families of Ricci-flat metrics in both situations.

For Lie algebras associated to a graph, we prove that an arrow-breaking involution, hence a Ricci-flat metric, always exists. This is obtained as a consequence of a more general result concerning the existence of arrow-breaking involutions on Lie algebras $\mgg$ with large center. In fact, we show that an arrow-breaking involution always exists when $\dim \mz(\g)\geq (\dim \g-3 )/2$, which always holds on two-step nilpotent Lie algebras associated to graphs.

For parabolic nilradicals, we obtain infinite families associated to $A_n$, $B_n$, $C_n$, as well as one example associated to $\Gtwo$. The corresponding Lie groups appear naturally as submanifolds of the symmetric spaces $\SL(n)/\SO(n)$, \linebreak $\SO(p,q)/\SO(p)\times \SO(q)$, $\Sp(n,\R)/\mathrm{U}(n)$ and $\Gtwo^*/\SO(4)$, as in~\cite{Tamaru}; however, due to the indefinite signature, our metrics do not extend to an invariant metric on the ambient space. Whilst we do not obtain Ricci-flat metrics on all the parabolic nilradicals, we emphasize that our method is constructive, and the resulting metrics are completely explicit.

\smallskip

\noindent {\bf Acknowledgements:}  V.~del Barco acknowledges the receipt of a grant from the ICTP-INdAM Research in Pairs Programme, Trieste, Italy. V.~del Barco and F.A.~Rossi acknowledge Université Paris-Sud and Università di Milano Bicocca who hosted them as visitors during part of the preparation of this work.
D.~Conti and F.A.~Rossi acknowledge GNSAGA of INdAM.

The authors express their gratitude to the anonymous referee for her/his useful suggestions.

\section{Nice Lie algebras and nice diagrams}
\label{sec:nice}
In this section we recall some basic definitions and language which will be used in the sequel.

Given a Lie algebra $\g$,  a basis  $\{e_1,\dotsc, e_n\}$ is called \emph{nice} if each $[e_i,e_j]$ is a multiple of a single basis element $e_k$ depending on $i,j$, and each $e_i\hook de^j$ is a multiple of a single $e^h$, depending on $i,j$; here, $\{e^1,\dotsc, e^n\}$ denotes the dual basis of $\g^*$. This definition was originally given in~\cite{LauretWill:Einstein}.

A \emph{nice Lie algebra} is a pair $(\g,\B)$, with $\g$ a Lie algebra and $\B$ a nice basis;  since componentwise rescaling preserves the nice basis condition, a nice Lie algebra  $(\g,\B)$  is regarded as equivalent to $(\g',\B')$ if there is a Lie algebra isomorphism $\g\cong \g'$ mapping basis elements to multiples of basis elements. This definition was used in~\cite{ContiRossi:Construction} to classify nice nilpotent Lie algebras of dimension~$\leq9$ up to equivalence. We point out that in this classification some Lie algebras appear with two inequivalent nice bases.

The properties of a nice basis are encoded in a combinatorial object called a nice diagram. Recall from~\cite{ContiRossi:Construction} that a \emph{labeled diagram} is a directed acyclic graph having no multiple arrows with the same source and destination, where each arrow is labeled with a node. A \emph{nice diagram} is a labeled diagram  $\Delta$ satisfying:
\begin{enumerate}[label=(N\arabic*)]
\item\label{enum:condNice1} any two distinct arrows with the same source have different labels;
\item\label{enum:condNice2} any two distinct arrows with the same destination have different labels;
\item\label{enum:condNice3} if $x\xrightarrow{y}z$ is an arrow, then $x$ differs from $y$ and $y\xrightarrow{x}z$ is also an arrow;
\item\label{enum:condNice4} there do not exist four nodes $x,y,z,w$ such that exactly one of
\[x\xrightarrow{y,z} w, \quad y\xrightarrow{z,x} w,\quad  z\xrightarrow{x,y} w\]
holds. The notation $x\xrightarrow{y,z} w$ means that $x,y,z$ are distinct nodes and there is a node $u$ such that $y\xrightarrow{z}u$, $x\xrightarrow{u}z$ belong to the diagram.
\end{enumerate}

Each nice nilpotent Lie algebra $(\g,\B)$ has an associated nice diagram $\Delta$ with set of nodes $N(\Delta)=\B$, obtained by declaring that $x\xrightarrow{y}z$ is an arrow if and only if $[x,y]$ is a nonzero multiple of $z$; one easily sees that~\ref{enum:condNice3} is a consequence of $[x,y]=-[y,x]$,~\ref{enum:condNice4} follows from the Jacobi identity, and~\ref{enum:condNice1}, \ref{enum:condNice2} follow from the definition of nice basis. The assumption that $\g$ is nilpotent is reflected in the fact that $\Delta$ is acyclic; definitions can be adjusted to work more generally, but we will not need to do so in the present paper.

There is a natural notion of isomorphism for nice diagrams, and it is clear that equivalent nice Lie algebras determine isomorphic nice diagrams. This is not a one-to-one correspondence, however:

\begin{itemize}
\item A nice diagram can have no associated Lie algebra. Consider for instance the nice diagram containing the nodes $e_1,\dotsc, e_7$ and the arrows
\[e_1\xrightarrow{e_2}e_3,\;
 e_1\xrightarrow{e_3}e_4,\;
 e_1\xrightarrow{e_4}e_5,\;
 e_2\xrightarrow{e_5}e_6,\;
 e_3\xrightarrow{e_4}e_6,\;
 e_1\xrightarrow{e_6}e_7,\;
 e_3\xrightarrow{e_5}e_7,
\]
together with their symmetric given by \ref{enum:condNice3}. A Lie algebra $\g$ with this diagram would have the form
\[(0,0,c_{123}e^{12},c_{134}e^{13}, c_{145}e^{14}, c_{256}e^{25}+c_{346}e^{34},c_{167}e^{16}+c_{357}e^{35}),\]
where the $c_{ijk}$ are nonzero constants. This notation means that relative to some basis $e^1,\dotsc, e^7$ of $\g^*$,  $de^1=de^2=0$, $de^3=c_{123}e^{12}$ and so on, where as usual $e^{12}$ is short for $e^1\wedge e^2$. It is straightforward to check that the condition $d^2=0$ is not satisfied, implying that there is no Lie algebra with this diagram (see also~\cite[Remark 1.7]{ContiRossi:Construction}).
\item A nice diagram can have more than one associated Lie algebra. For instance, the nice nilpotent Lie algebras
\begin{align*}
\texttt{6431:2a}&\quad(0,0,e^{12},e^{13},e^{23},e^{14}+e^{25})\\
\texttt{6431:2b}&\quad(0,0,e^{12},- e^{13},e^{23},e^{14}+e^{25})
\end{align*}
are not equivalent (in fact, they are not even isomorphic, see~\cite{Magnin}). The string \texttt{6431:2a} refers to the name given to the nice Lie algebra in the classification of~\cite{ContiRossi:Construction}, where the part before the colon represents the dimensions in the lower central series and the number after the colon is a progressive number, possibly followed by a letter to identify inequivalent nice Lie algebras associated to the same diagram.
\item A nice diagram can in fact be associated to infinitely many Lie algebras, consider e.g. the Lie algebra \texttt{754321:9}
\[(0,0,(1-\lambda) e^{12},e^{13},\lambda e^{14}+e^{23},e^{24}+e^{15},e^{34}+e^{25}+e^{16}), \quad \lambda\neq0,1,\]
corresponding to the one-parameter family of nilpotent Lie algebras $123457I$ in the classification of~\cite{Gong}.
\end{itemize}

It is also worth pointing out that not all nilpotent Lie algebras admit a nice basis. The number of nice nilpotent Lie algebras taken up to equivalence by dimension is summarized in Table~\ref{table:classification} together with the number of nilpotent Lie algebras taken up to isomorphism, making evident the nonrestrictiveness of the nice condition. The semiinteger entry in dimension $7$ reflects the fact that one of the families appearing in the classification of~\cite{Gong} has a nice basis only for positive values of the real parameter; the question marks in higher dimensions reflects the lack of a classification for nilpotent Lie algebras beyond dimension $7$ and the fact that, lacking such a classification, finding whether two nice Lie algebras are isomorphic is a nontrivial problem.

\begin{table}[thp]
\centering
{\caption{\label{table:classification} Number of nilpotent Lie algebras (NLA) and nice nilpotent Lie algebras by dimension, according to the classifications of \cite{Gong} and \cite{ContiRossi:Construction}.}
\begin{tabular}{CCCC}
\toprule
\text{dim} & \text{NLA} & \text{NLA admitting nice basis} & \text{nice NLA}\\
\midrule
 3&2&2&2\\
 4&3&3&3\\
 5&9&9&9\\
 6&34&33&36\\
 7&175 + 9 \text{ families } & 141 +4\frac12 \text{ families}&152 +4\frac12 \text{ families} \\
 8&?&? &917+45 \text{  families } \\
 9&?&? &6386  +501  \text{  families } \\
\bottomrule
\end{tabular}
}
\end{table}

\FloatBarrier

\section{$\sigma$-diagonal metrics}
\label{sec:sigma}
In this section we consider left-invariant pseudoriemannian metrics on a nilpotent Lie group; these will be expressed as metrics (i.e. indefinite scalar products) on the corresponding Lie algebra. We give a formula for the curvature of the associated Levi-Civita connection. We introduce a particular class of Ricci-flat metrics and produce sufficient conditions to prove that they are not flat.

Let $\mcB=\{e_1, \ldots, e_n\}$ be a basis of a Lie algebra $\mg$ and let $\sigma$ be an order two permutation of $\mcB$; we will write $e_{\sigma_i}$ for $\sigma(e_i)$. Having numbered the elements of $\mcB$, we shall represent $\sigma$ as a product of transpositions in $\{1, \ldots, n\}$. A \emph{$\sigma$-diagonal} metric $\bil$ on $\mg$ with respect to the basis $\mcB$ is a metric satisfying
\begin{equation}\label{eqn:sigmametric}
\lela e_i, e_{\sigma_j}\rira=g_i\delta_{ij},
\end{equation}
where $g_i$ are nonzero reals satisfying $g_{\sigma_i}=g_i$. As the $g_i$  vary, we obtain for the signature any pair $(p,q)$ such that  $\abs{p-q}$ does not exceed the number of nodes fixed by $\sigma$.

If $\mg$ is a nice nilpotent Lie algebra with diagram $\Delta$, we will consider metrics that are $\sigma$-diagonal with respect to the nice basis, so that $\sigma$ becomes an order two permutation of the set of nodes $N(\Delta)$; for fixed $\g$ and $\sigma$,
\eqref{eqn:sigmametric} defines a family of $\sigma$-diagonal metrics. To count the number of parameters describing  these families, we represent Lie algebras of dimension $n$ by elements of $\Lambda^2(\R^n)\otimes\R^n$ corresponding to the Lie bracket and declare the metric to be a fixed scalar product $\bil$  on $\R^n$ of the type~\eqref{eqn:sigmametric}. Two elements $c,c'$ of $\Lambda^2(\R^n)\otimes\R^n$ correspond to isometric Lie algebras if and only if they are related by the action of the orthogonal group $O(\bil)$; indeed, an element $g\in \GL(n,\R)$ corresponds to an isomorphism between the Lie algebras determined by $c$ and $c'$ if and only if $c'=gc$.

Recall that to a nice diagram $\Delta$ with nodes $e_1,\dotsc, e_n$ one can associate a \emph{root matrix} with $n$ columns $M_\Delta$ such that whenever $e_i\xrightarrow{e_j}e_k$ is an arrow $M_\Delta$ has a row with $-1$ in the entries $i,j$, $+1$ in the entry $k$, and zero in the others.

\begin{proposition}
\label{prop:family}
Let $(\g,\B)$ be a nice Lie algebra and let $\sigma$ be an order two permutation of $\B$ which is the product of $k$ transpositions. Let
$(\ker M_\Delta)^{-\sigma}$ be the subspace of $\ker M_\Delta$ consisting of vectors $X$ that satisfy $\sigma(X)=-X$.

Then $\sigma$-diagonal metrics on $(\g,\B)$ form  a family of nonisometric metric Lie algebras depending on
\[\operatorname{rank} M_\Delta-
k+ \dim (\ker M_\Delta)^{-\sigma}\]
parameters.
\end{proposition}
\begin{proof}
The family of metric Lie algebras corresponding to $\sigma$-diagonal metrics on $\g$, as the parameters $g_i$ vary, form a $D_n$-orbit in $(\Lambda^2\R^n)\otimes\R^n$, where $D_n$ is the Lie group of nonsingular diagonal matrices. The stabilizer for this action can be identified with $\ker M_\Delta$; thus, the $D_n$-orbit is a submanifold of \mbox{$V_\Delta=\Span{e^{ij}\otimes e_k\mid i\xrightarrow{j}k \text{ is an arrow}}$} with dimension equal to $\operatorname{rank} M_\Delta$.

Two elements determine isometric metric Lie algebras if they are in the same $O(\langle,\rangle)$-orbit. By the same argument used in the proof of \cite[Theorem 3.6]{ContiRossi:Construction}, the
$O(\langle,\rangle)$-orbit intersected with $V_\Delta$  has the same tangent space as the $D_n\cap O(\langle,\rangle)$-orbit.

The Lie algebra of $D_n\cap O(\langle,\rangle)$ has dimension $k$; it can be identified with the set of elements $X\in\R^n$ such that $\sigma(X)=-X$ by taking the diagonal matrices with the same entries, and the stabilizer with $(\ker M_\Delta)^{-\sigma}$. Thus, the $D_n\cap O(\langle,\rangle)$-orbit has dimension $k-\dim (\ker M_\Delta)^{-\sigma}$.
\end{proof}

Notice that unlike in~\cite{ContiRossi:EinsteinNice,ContiRossi:Construction,ContiRossi:RicciFlat}, we do not require $\sigma$ to be an automorphism of the diagram associated to the nice Lie algebra, i.e. to map arrows to arrows; in fact, the relevant condition for this paper implies that the image of an arrow is never an arrow:
\begin{definition}\label{def:arrowbreaking}
Given a nice diagram $\Delta$, a permutation of its set of nodes will be called an \emph{arrow-breaking involution} if it has order two and:
\begin{enumerate}
 \item whenever $x$ has an incoming arrow with label $y$, then $\sigma(x)$ does not have an incoming arrow with label $\sigma(y)$;
 \item whenever $x$ has an outgoing arrow with label $y$, then $\sigma(x)$ does not have an outgoing arrow with label $\sigma(y)$.
\end{enumerate}
\end{definition}

\begin{proposition}
\label{prop:2OrderSigmaEnhancedGiveRicciFlat}
Let $\g$ be a nice nilpotent Lie algebra with diagram $\Delta$, and let $\sigma$ be an arrow-breaking involution.
Then any $\sigma$-diagonal metric~\eqref{eqn:sigmametric} is Ricci-flat.
\end{proposition}
\begin{proof}
By~\cite{ContiRossi:EinsteinNice}, we have
\begin{equation}
 \label{eqn:ricddadad}
\ric(v,w) = \frac12\langle dv^\flat, dw^\flat\rangle-\frac12\langle \ad v,\ad w\rangle,
\end{equation}
where $v^\flat=\langle v,\cdot\rangle$. It therefore suffices to show that the metric restricts to zero on the spaces
\[\ad \g=\{\ad x\st x\in\g\}, \quad d\g^*=\{d\alpha\st \alpha\in\g^*\}.\]
Let $\{e_1, \ldots,e_n\}$ be a nice basis of $\mg$; by the nice condition, $\ad\g$ is spanned by the elements $e^i\otimes e_k$ such that $e_i\xrightarrow{e_j}e_k$ is an arrow in $\Delta$.

Assume therefore that $e_i\xrightarrow{e_j}e_k$ is an arrow in $\Delta$. By the form of the metric,
\[\langle e^i\otimes e_k ,e^r\otimes e_p\rangle = \delta_{\sigma_ir }\delta_{\sigma_kp}g_k/g_i.\]
Thus, $e^i\otimes e_k$ is orthogonal to $\ad\g$ unless $e^{\sigma_i}\otimes e_{\sigma_k}$ also belongs to $\ad\g$, i.e.
\[e_{\sigma_i}\xrightarrow{e_h}e_{\sigma_k}\]
is an arrow for some $h$.  Similarly, we have
\[d\g^*\subset \Span{e^{ij}\mid e_i\xrightarrow{e_j}e_k},\]
where
\[\langle e^{ij},e^{rp}\rangle =\det
 \begin{pmatrix} \delta_{\sigma_i r}/(g_ig_r) & \delta_{\sigma_i p}/(g_ig_p)\\ \delta_{\sigma_j r}/(g_jg_r) & \delta_{\sigma_j p}/(g_jg_p)\end{pmatrix}.
\]
Suppose $e_i\xrightarrow{e_j}e_k$ is an arrow in $\Delta$. Then $e^{ij}$ is orthogonal to $d\g^*$ unless $\sigma_i\xrightarrow{\sigma_j} h$ is an arrow in $\Delta$ for some $h$; this is absurd.
\end{proof}
\begin{remark}
 \label{remark:structureconstantsdonotmatter}
We point out that the arrow-breaking condition only depends on the underlying diagram of a Lie algebra, rather than the Lie algebra. This means that the actual structure constants do not play any role, as long as we only consider Ricci-flat metrics of the particular type~\eqref{eqn:sigmametric}.
\end{remark}

\begin{example} Consider the nice Lie algebra
\begin{equation*}
\texttt{52:1}\quad(0,0,0,e^{12},e^{13}).
\end{equation*}
It is easy to check that the order two permutation $\sigma= (3\, 4)(2\, 5)$ is  arrow-breaking. Therefore, the metric
\begin{equation}\label{eq:metex1}
\langle e_1,e_1\rangle = g_1,\quad \langle e_2,e_5\rangle=g_2,\quad \langle e_3,e_4\rangle= g_3
\end{equation}
is Ricci-flat for any choice of the parameters $g_i$ by Proposition~\ref{prop:2OrderSigmaEnhancedGiveRicciFlat}.

In addition, the root matrix $M_\Delta$ has rank two, $k=2$ and $(\ker M_\Delta)^{-\sigma}$ is spanned by $ (0,1,-1,1,-1)$.  Proposition~\ref{prop:family} implies that \eqref{eq:metex1} gives a family of nonisometric Ricci-flat metric Lie algebras depending on $2-2+1=1$ parameter.
\end{example}


Recall from~\cite{AlekseevskiKimelFel} that, in the  Riemannian case, homogeneous Ricci-flat ma\-nifolds are necessarily flat. In the pseudoriemannian context, this is not true; therefore, we are interested in determining whether a Ricci-flat metric is flat or not. To this end, we generalize to our less restrictive setting a formula for the Riemann tensor of a metric Lie algebra proved by Boucetta~\cite{Boucetta} in the two-step nilpotent case. For this computation, we do not assume that the Lie algebra is nice or nilpotent.

Let $\mg$ be a Lie algebra and let $\bil$ be a metric on it. Fix a  basis $\{e_1, \ldots, e_p\}$ of the commutator $\mg'=[\mg,\mg]$ of $\mg$, and consider a linearly independent set $\{x_1, \ldots, x_p\}\subset \mg$ such that
\[\lela e_i,x_j\rira=\delta_{ij};\]
for instance, the $x_i$ can be constructed by completing $\{e_1, \ldots, e_p\}$ to a basis of $\g$ and taking the metric duals of the dual basis.

For $i=1, \ldots,p$, define the endomorphisms $J_i:\mg\to \mg$ by
\begin{equation}
\label{eqn:defJi}
\lela J_iu,v\rira=\lela [u,v],x_i\rira, \quad \mbox{ for } u,v\in \mg.
\end{equation}
Notice that $J_i$ is a skew-symmetric endomorphism of $(\mg,\bil)$. Moreover, if $\mz$ denotes the center of $\mg$, that is, $\mz=\{z\in \mg\st [x,z]=0 \mbox{  for all }x\in \mg\}$, then we have $J_i\mz=0$, $J_i\mg\subseteq \mz^\bot$ and $\bigcap_{i=1}^p \ker J_i=\mz$.

From~\eqref{eqn:defJi}, the Lie bracket can be written in terms of the skew-symmetric endomorphisms: for every $u,v\in \mg$,
\begin{equation}
\label{eq.bracket}
[u,v]=\sum_{i=1}^p\lela J_iu,v\rira e_i.
\end{equation}
The Levi-Civita connection of $(\mg,\bil)$ has the following expression:
\begin{equation}\label{eq:Koszul}
2\nabla_uv=\sum_{i=1}^p (\lela J_iu,v\rira e_i-\lela e_i,v\rira  J_iu-\lela e_i,u\rira J_iv), \qquad u,v\in \mg;
\end{equation}
we deduce the following for the curvature tensor
$R(u,v)=\nabla_{[u,v]}-[\nabla_u,\nabla_v]$:
\begin{proposition}
\label{prop:riemann}
Given a metric Lie algebra $\g$ and $J_i$ as above, for all  $u,v,w$ in $\mg$ the curvature tensor satisfies
\begin{multline}\label{eq:curv}
R(u,v)w=\sum_{i,j=1}^p\lela e_i,e_j\rira \left( \frac14 \lela J_i v,w \rira J_ju-\frac14 \lela J_iu,w\rira J_jv-\frac12\lela J_i  u,v\rira J_j w\right)\\
+ \frac14\sum_{i,j=1}^p\left(
\lela e_i,w\rira \lela e_j,v\rira J_j\circ J_iu-\lela e_i,w\rira\lela e_j,u\rira J_j\circ J_i v+\lela e_i,u\rira\lela e_j,v\rira [J_j,J_i]w
\right)\\
+\frac14\sum_{i,j=1}^p\left(\lela e_i,w\rira \lela [J_j,J_i]u,v\rira+\lela e_i,v\rira\lela J_ju,J_iw\rira-\lela e_i,u\rira\lela J_jv,J_iw\rira
\right)e_j\\
-\frac14[w,[u,v]]+\sum_{j=1}^p\left(
-\frac12 \lela w,e_j\rira  J_j[u,v]+\frac14\lela e_j,u\rira J_j[v,w]-\frac14 \lela v,e_j\rira J_j[u,w]\right)\\
-\frac14 \sum_{j=1}^p\left(\lela [v,e_j],w\rira+\lela [w,e_j],v\rira\right) J_ju+\frac14 \sum_{j=1}^p \left(\lela [u,e_j],w\rira +\lela [w,e_j],u\rira\right)J_jv.
\end{multline}
\end{proposition}
We make use of this general expression to compute the sectional curvature of a $\sigma$-diagonal metric~\eqref{eqn:sigmametric} on a Lie algebra.

Let $\mg$ be a Lie algebra and fix a basis $\{e_1, \ldots, e_p, e_{p+1},\ldots e_n\}$ such that the first $p$ elements span $\mg'$. Pick a $\sigma$-diagonal metric satisfying $\lela e_i,e_j\rira=g_i \delta_{i\sigma_j}$; thus,
\begin{equation}
\label{eqn:Jnice}
\lela J_iu,v\rira=\frac1{g_i}\lela [u,v],e_{\sigma_i}\rira,\quad \mbox{ for all } i=1, \ldots, p.
\end{equation}
Using the general formula~\eqref{eq:curv}, we obtain the following useful expression for the basis elements of the Lie algebra: for $s,t=1, \ldots, n$ we have
\begin{equation}
\label{eqn:sc}
\begin{split}
\lela R(e_s,e_t)e_s,e_t\rira &= \frac32\sum_{i=1}^pg_i  \lela J_i e_t,e_s \rira \lela J_{\sigma_i}e_s,e_t\rira
+g_s g_t
\lela J_{\sigma_t}\circ J_{\sigma_s}e_s,e_t\rira\\
&\quad-\frac12 g_sg_t
 \lela J_{\sigma_s}J_{\sigma_t}e_s,e_t\rira
-\frac14\left(
g_t^2\lela J_{\sigma_t}^2e_s,e_s\rira+
 g_s^2\lela J_{\sigma_s}^2 e_t,e_t\rira
\right)\\
&\quad-\frac14\lela[e_s,[e_s,e_t]],e_t\rira+\frac34  \lela [[e_t,e_s],e_t],e_s\rira\\
&\quad-\frac14 \sum_{j=1}^p\left(\lela [e_t,e_j],e_s\rira+\lela [e_s,e_j],e_t\rira\right) \lela J_je_s,e_t\rira.
\end{split}
\end{equation}
Recall that the sectional curvature of a pseudoriemannian metric is defined on non-degenerate planes. Nevertheless, as in the Riemannian case, the metric $(\mg,\bil)$ is flat if and only if for  every $s,t=1,\ldots,n$ one has $\lela R(e_s,e_t)e_s,e_t\rira=0$ (see~\cite[Chapter 3]{oneill:semiriemannian}).

Assume now that $\mg$ is a nice nilpotent Lie algebra and $\mcB=\{e_1, \ldots, e_n\}$ as above is a nice basis. Given $e_s,e_t\in \mcB$, the nice condition implies that there  exist
$\lambda_1\in \R$ and $e_{k_1}\in \mcB\cup \{e_{\infty}\}$ such that
\[
[e_s,e_t]=\lambda_1 e_{k_1},
\]
with the convention that $k_1=\infty$  and $\lambda_1=0$ if $e_s,e_t$ commute. Similarly,
there  exist $\lambda_i,\mu_j\in \R$ with $
i=2, \ldots, 5$, $j=1,2$ and $e_{k_i}\in \mcB\cup \{e_\infty\}$, $i=2 ,\ldots,5$ such that
\begin{equation}\label{eqn:ks}
\begin{gathered}
[e_s,e_{k_2}]=\lambda_2 e_{\sigma_t},\quad [e_t,e_{k_3}]=\lambda_3 e_{\sigma_s},\\
[e_s,e_{k_4}]=\mu_1e_{\sigma_s},\quad [e_t,e_{k_5}]=\mu_2e_{\sigma_t}.
\end{gathered}
\end{equation}
It will be understood that $\sigma_\infty=\infty$ and $\delta_{ij}=0$ when either $i$ or $j$ is $\infty$. The endomorphisms $J_i$ have an explicit formula in this case, by using~\eqref{eqn:Jnice}, so we obtain
\begin{align*}
\lela J_{r} e_s,e_t\rira&=\frac{\lambda_1}{g_r} \delta_{r,k_1}, \\
\lela J_{\sigma_t}J_{\sigma_s}e_s,e_t\rira &=-\frac{\mu_1\mu_2}{g_{k_4}} \delta_{k_5 ,\sigma_{k_4}},&
\lela J_{\sigma_s}J_{\sigma_t}e_s,e_t\rira &=-\frac{\lambda_2\lambda_3}{g_{k_2}}\delta_{k_3,\sigma_{k_2}}, \\
\lela J_{\sigma_t}^2e_s,e_s\rira&=-\frac{\lambda_2^2}{g_{k_2}} \delta_{k_2,\sigma_{k_2}},
& \lela J_{\sigma_s}^2e_t,e_t\rira&=-\frac{\lambda_3^2}{g_{k_3}}\delta_{k_3,\sigma_{k_3}}, \\
\lela [e_s,[e_s,e_t]],e_t\rira&=\lambda_1\lambda_2g_t\delta_{k_1,k_2}, &
\lela [[e_t,e_s],e_t],e_s\rira&=\lambda_1\lambda_3g_s\delta_{k_1,k_3}, \\
\lela [e_t,e_j],e_s\rira&=g_s\lambda_3\delta_{j,k_3}, &
\lela [e_s,e_j],e_t\rira&=g_t\lambda_2\delta_{j,k_2},\\
\lela J_{k_3}e_s,e_t\rira&=\lambda_1 \delta_{k_1,k_3}, &
\lela J_{k_2}e_s,e_t\rira&= \lambda_1\delta_{k_1,k_2}.
\end{align*}
Therefore,~\eqref{eqn:sc} becomes
\begin{equation}\label{eqn:scnice}
\begin{aligned}
\lela R(e_s,e_t)e_s,e_t\rira
&=-\frac{3}{ g_{k_1}}\lambda_1^2\delta_{k_1,\sigma_{k_1}}
+\frac14\lambda_2^2\frac{g_t^2}{g_{k_2}}\delta_{k_2,\sigma_{k_2}}
+\frac14\lambda_3^2 \frac{g_s^2}{g_{k_3}}\delta_{k_3,\sigma_{k_3}}\\
&\quad
-\frac12\lambda_1\lambda_2g_t\delta_{k_1,k_2}
+\frac12 \lambda_1\lambda_3g_s \delta_{k_1,k_3}
\\
&\quad+\frac12 \lambda_2\lambda_3\frac{g_sg_t}{g_{k_2}}\delta_{k_3,\sigma_{k_2}}
-\mu_1\mu_2 \frac{g_s g_t}{g_{k_4}}\delta_{k_5 ,\sigma_{k_4}}.
\end{aligned}
\end{equation}
We deduce straight from~\eqref{eqn:scnice}:

\begin{proposition}\label{pr:CriteriaFlatness}
Let $\mg$ be a nice nilpotent Lie algebra with nice basis $\mcB=\{e_1, \ldots, e_n\}$ and let $\sigma$ be an order two permutation of the basis. Suppose that for some $e_s,e_t\in \mcB$, there is no $k$ verifying both $e_s\xrightarrow{e_k}e_{\sigma_s}$ and $e_t\xrightarrow{e_{\sigma_{k}}}e_{\sigma_t}$. If at least one of the following conditions holds, then every $\sigma$-diagonal metric~\eqref{eqn:sigmametric} is nonflat:
\begin{enumerate}[label=(C\arabic*)]
\item\label{enum:CriteriaFlat1} $e_s\xrightarrow{e_t}e_{k_1}$ with $k_1$ fixed by $\sigma$, and there are no arrows of the form $e_t\xrightarrow{\bullet} e_{\sigma_s}$ or $e_s\xrightarrow{\bullet} e_{\sigma_t}$; 
\item\label{enum:CriteriaFlat2} $e_s\xrightarrow{k_2}e_{\sigma_t}$ with $k_2$  fixed by $\sigma$, and there are no arrows of the form $e_s\xrightarrow{e_t}\bullet$ or $e_t\xrightarrow{\bullet} e_{\sigma_s}$. 
\end{enumerate}
\end{proposition}

\begin{remark}
Recall that the arrow-breaking condition only depends on the underlying diagram of a Lie algebra, rather than the Lie algebra (see Remark \ref{remark:structureconstantsdonotmatter}).
The above criteria for non-flatness are also independent of the structure constants.
Nevertheless, the full curvature tensor of a metric induced by an arrow-breaking permutation depends on the structure constants.
\end{remark}

\begin{remark}
\label{remark:flat}
Given an arrow-breaking $\sigma$, all the $\sigma$-diagonal metrics as in~\eqref{eqn:sigmametric}
are Ricci-flat, regardless of the parameters $g_i$ (Proposition~\ref{prop:2OrderSigmaEnhancedGiveRicciFlat}); however, the full curvature tensor may or may not depend on the parameters.
For instance, the arrow-breaking involution $\sigma= (3\, 4)(2\, 5)(1\, 6)$ of the Lie algebra
\[\texttt{64321:4}\quad(0,0,e^{12},e^{13},e^{14}+e^{23},e^{15}+e^{24})\]
gives the metric
\[
\langle e_1,e_6\rangle = g_1,\quad \langle e_2,e_5\rangle=g_2,\quad \langle e_3,e_4\rangle= g_3.\]
 By Proposition~\ref{prop:riemann}, the curvature tensor of this metric is
\[ \frac{ g_1- g_3}{g_3} e^{12}\otimes e^1\otimes e_4+ \frac{  g_1-  g_3}{ g_2}e^{13}\otimes e^1\otimes e_5-\frac{  g_1-  g_3}{g_1}(e^{12}\otimes e^3\otimes e_6+e^{13}\otimes e^2\otimes e_6).\]
Clearly, it is flat if and only if $g_1=g_3$.

In the notation of  Proposition~\ref{prop:family}, $M_\Delta$ has kernel spanned by $(1,2,3,4,5,6)$, hence rank $5$, and $(\ker M_\Delta)^{-\sigma}=0$. So we obtain a family of nonisometric Ricci-flat metric Lie algebras with $5-3=2$ parameters, within which we find a one-parameter family of flat Lie algebras.

However, in the Lie algebra
\[\texttt{61:1}\quad(0,0,0,0,0,e^{12})\]
the involution $\sigma= (1\, 6)$ gives the $\sigma$-diagonal metric
\[
\langle e_1,e_6\rangle = g_1,\quad \langle e_2,e_2\rangle=g_2,\quad \langle e_3,e_3\rangle= g_3
,\quad \langle e_4,e_4\rangle= g_4,\quad \langle e_5,e_5\rangle= g_5.
\]
In this case, $M_\Delta$ has rank one, $k=1$ and $(\ker M_\Delta)^{-\sigma}=0$, so all these metrics are isometric. A direct computation shows that they are flat.
\end{remark}

Given a diagram $\Delta$, it will be convenient to consider the ring $\Z[{\mathrm{e}}_1,\dotsc, {\mathrm{e}}_n]$, where each indeterminate ${\mathrm{e}}_i$ is associated to the node $e_i$ of $\Delta$.
Let $P_\Delta,Q_\Delta\in\Z[{\mathrm{e}}_1,\dotsc, {\mathrm{e}}_n]$ be the polynomials
\begin{equation}\label{eqn:PQ}
P_\Delta=\prod_{\substack{e_i\xrightarrow{e_j}e_k\\ i<j}} ({\mathrm{e}}_i+{\mathrm{e}}_j), \qquad Q_\Delta=\prod_{e_i\xrightarrow{e_j}e_k} (1+{\mathrm{e}}_i-{\mathrm{e}}_k).
\end{equation}
We shall refer to the degree-one polynomials $({\mathrm{e}}_i+{\mathrm{e}}_j)$, $(1+{\mathrm{e}}_i-{\mathrm{e}}_k)$ as the linear factors of $P_\Delta Q_\Delta$.

Notice that $Q_\Delta$ does not depend on the labels of the arrows, i.e. it is associated to the underlying unlabeled diagram.
\begin{example}

Let $\mg_n$ denote the standard filiform Lie algebra of dimension $n$. Such a Lie algebra has a nice basis $\{e_1,\ldots, e_n\}$ satisfying the nonzero Lie bracket relation
\[ [e_1,e_i]=e_{i+1}, \quad i=2, \ldots, n-1.\]
The corresponding polynomials are
\[P_\Delta=\prod_{i=2}^{n-1}({\mathrm{e}}_1+{\mathrm{e}}_i), \quad Q_\Delta=\prod_{i=2}^{n-1}(1+{\mathrm{e}}_1-{\mathrm{e}}_{i+1})(1+{\mathrm{e}}_{i}-{\mathrm{e}}_{i+1}).\]

For $n\geq 5$, the standard filiform Lie algebra $\mg_n$ possesses a Ricci-flat nonflat metric. Indeed, the order two permutation defined by $\sigma(e_i)=e_{n-i+1}$, for $i=1, \ldots, n$ 
is arrow-breaking. Moreover, for $n\geq 5$,
the curvature satisfies
\[R(e_1,e_2)e_{n-3}
=\frac{1}{2g_1}(g_3+g_4)e_n,\] so it is nonzero if $g_3\neq -g_4$. 


\end{example}

There is a natural action of $\Sigma_n$, the group of permutations of $\{1, \ldots, n\}$, on $\Z[{\mathrm{e}}_1,\dotsc, {\mathrm{e}}_n]$, for which we trivially have
\[\sigma P_\Delta=P_{\sigma(\Delta)}, \quad \sigma Q_\Delta=Q_{\sigma(\Delta)}.\]

\begin{lemma}
\label{lemma:arrowbreaking}
Given a nice diagram $\Delta$ and an order two  permutation  $\sigma\colon N(\Delta)\to N(\Delta)$, the following are equivalent:
\begin{enumerate}[label=\arabic*)]
 \item\label{enum:CondLemmaArrowbreaking1} $\sigma$ is arrow-breaking;
 \item\label{enum:CondLemmaArrowbreaking2} $P_\Delta$ and $\sigma P_\Delta$ have no linear factors in common and $Q_\Delta$ and $\sigma Q_\Delta$ have no linear factor in common;
 \item\label{enum:CondLemmaArrowbreaking3} $P_\Delta$ and $Q_\Delta$ have no $\sigma$-invariant divisor of positive degree.
\end{enumerate}
\end{lemma}
\begin{proof}
The equivalence of~\ref{enum:CondLemmaArrowbreaking1} and~\ref{enum:CondLemmaArrowbreaking2} is obvious from the definition.

Clearly, a $\sigma$-invariant divisor of $P_\Delta$ of positive degree is a common divisor to $P_\Delta$ and $\sigma P_\Delta$, decomposing into the product of linear factors dividing  both polynomials.

Conversely, if ${\mathrm{e}}_i+{\mathrm{e}}_j$ divides both $P_\Delta$ and $\sigma P_\Delta$, then $({\mathrm{e}}_i+{\mathrm{e}}_j)({\mathrm{e}}_{\sigma_i}+{\mathrm{e}}_{\sigma_j})$ is a divisor of $P_\Delta$. Similarly for $Q_\Delta$.
\end{proof}

We will abuse terminology and write that  $P_\Delta$ and $Q_\Delta$ have no $\sigma$-invariant divisor when the equivalent conditions of Lemma~\ref{lemma:arrowbreaking} hold.

\begin{example}
Consider the Heisenberg Lie algebra $\lie{h}_{2n+1}$ with basis $\{e_1,\dotsc,\allowbreak e_{2n},y\}$ and non-zero Lie brackets $[e_{2i},e_{2i-1}]=y$, $i=1, \ldots, n$. Then
\[P_\Delta=({\mathrm{e}}_1+{\mathrm{e}}_2)\dotsm ({\mathrm{e}}_{2n-1}+{\mathrm{e}}_{2n}) \quad \mbox{and}\quad
 Q_\Delta=(1+{\mathrm{e}}_1-{\mathrm{y}})\dotsm (1+{\mathrm{e}}_{2n}-{\mathrm{y}}).
\]
Then $\sigma$ defined as
\[\sigma(e_1)=y, \quad \sigma(e_{2i})=e_{2i+1},\ i\geq2, \quad  \sigma(e_{2n})=e_{2n}\]
does not leave any divisor of $P_\Delta Q_\Delta$ invariant; therefore, $\sigma$ defines a Ricci-flat metric on $\mh_{2n+1}$. For $n=1$, a direct computation shows that the metric is flat. It is known from \cite{AM03} that Heisenberg Lie algebras do not admit flat  metrics for $n\geq2$. We can check that our metrics are not flat by using criterion~\ref{enum:CriteriaFlat2} in Proposition~\ref{pr:CriteriaFlatness} applied to $e_{2n-1}$ and $e_1$: in fact
$e_{2n-1}\xrightarrow{e_{k_2}}\sigma(e_1)=y$ with $e_{k_2}=e_{2n}$ fixed by $\sigma$, and there are no arrows of the form $e_{2n-1}\xrightarrow{e_1}\bullet$, $e_1\xrightarrow{\bullet} \sigma(e_{2n-1})=e_{2n-2}$.

The signature of this metric is $(n,n+1)$; other signatures can be obtained, for instance declaring
\[\sigma(e_1)=y, \quad \sigma(e_{4i})=e_{4i+2},\ i\geq2\]
when $n$ is odd. To prove that the metric is not flat for $n> 1$, we can apply again  criterion~\ref{enum:CriteriaFlat2} in Proposition~\ref{pr:CriteriaFlatness} to $e_4,e_1$: the arrow $e_4\xrightarrow{e_{k_2}}\sigma(e_1)=y$ with $e_{k_2}=e_3$ is fixed by $\sigma$ and there are no arrows of the form $e_4\xrightarrow{e_1}\bullet$, $e_1\xrightarrow{\bullet} \sigma(e_4)=e_6$.

It was proved in~\cite{Boucetta} that $\lie{h}_{2n+1}$ admits a Ricci-flat left-invariant metric of any signature $(q,2n+1-q)$ for $2\leq q\leq n$.
\end{example}

\section{Involutions on Lie algebras with large center}
\label{sec:involutions}
Given a  nilpotent Lie algebra $\mg$ with center $\mz$, we denote $s:=\dim \mz$ and $r:=\dim (\mg/\mz)$.  This terminology is adopted along the section in order to give sufficient conditions on $r,s$ for a Lie algebra to carry an arrow-breaking involution.

Let $\mg^i$ denote the lower central series of $\mg$, that is, $\mg^0=\mg$ and $\mg^i=[\mg,\mg^{i-1}]$ for $i\geq 1$. By an inductive reasoning one can prove
\begin{equation}
 \label{eqn:ij}
[\mg^{i},\mg^{j}]\subseteq \mg^{i+j+1},\qquad \mbox{ for  every } i,j\geq 0.
\end{equation}
Recall that if $\mg$ is $k$-step nilpotent, then $\mg^{k+1}=0$ and $\mg^k\subseteq \mz$. Set $\mb^i:=\mz+\mg^i$; then for every $i=1, \ldots, k-1$ one has $\mg^i\subsetneq \mb^{i+1}$, since $[\g,\mb^{i+1}]=\mg^{i+2}$.

Given a nilpotent Lie algebra $\mg$ with a nice basis $\mcB$, the basis is adapted to the lower central series; indeed, for each $i\geq 0$ there is a subset $\mcB_i$ of $\mcB$ such that $\mcB_i$ is a basis of $\mg^i$. In addition, any nice basis is a union of disjoint subsets $\mcB=\mcX\cup \mcZ$ where $\mcZ$ is a basis of $\mz$ and  $\mcX$ is a basis of a complement. In particular, $\mcX$ has $r$ elements and $\mcZ$ has $s$. Suppose that $\mcX=\{x_1,\dotsc, x_r\}$ and $\mcZ=\{z_1,\dotsc, z_s\}$; then, by~\eqref{eqn:PQ}, the linear factors of $P_\Delta$ have the form ${\mathrm{x}}_i+{\mathrm{x}}_j$ and the linear factors of $Q_\Delta$ have the form $1+{\mathrm{x}}_i-{\mathrm{z}}_j$, $1+{\mathrm{x}}_i-{\mathrm{x}}_j$.

\begin{proposition}
\label{prop:rs}
Let $\mg$ be a nice $k$-step nilpotent Lie algebra such that
\[r\leq  s+3.\]
Then $\g$ has an arrow-breaking involution $\sigma$.
\end{proposition}
\begin{proof}
Let $\mcB=\mcX\cup \mcZ$ be a nice basis of $\mg$ as above. If $r\leq s$, we can choose an involution $\sigma\colon \mcB\to \mcB$ which is the extension of an injective map $\sigma:\mcX\to \mcZ$; it is clear that $P_\Delta Q_\Delta$ has no $\sigma$-invariant divisor.

If $r>s$, we consider involutions satisfying
\begin{equation}
 \label{eqn:sigmamapsxitoyi}
 \sigma(\mcZ)=\mcX.
\end{equation}
If $x_i=\sigma (z_i)$ for $i=1, \ldots, s$ and $\mcX=\{x_1, \ldots, x_s, y_1, \ldots,y_{r-s}\}$,
it is clear that invariant divisors of $Q_\Delta$ arise from factors of the form $1+{\mathrm{y}}_i-{\mathrm{x}}_j$, and any invariant divisor of $P_\Delta$ will be in $\Z[{\mathrm{y}}_{1},\dotsc, {\mathrm{y}}_{r-s}]$. We will show that it is possible to choose $y_1,\ldots, y_{r-s}\in \mcX$  so that $P_\Delta Q_\Delta$ has no $\sigma$-invariant divisors.
\smallskip

For $r=s+1$, take $\sigma$ as in~\eqref{eqn:sigmamapsxitoyi}  fixing some $y_1\in\mcX$. Then, $ P_\Delta Q_\Delta$ has no invariant divisor under any $\sigma$ of the above type, since $\sigma({\mathrm{y}}_{1}+{\mathrm{x}}_i)={\mathrm{y}}_{1}+{\mathrm{z}}_i$ and $\sigma(1+{\mathrm{y}}_1-{\mathrm{x}}_i)=1+{\mathrm{y}}_1-{\mathrm{z}}_i$.
\smallskip

For $r=s+2$ and step two, it suffices to choose  $y_1,y_2\in \mcX$ in such a way that $ {\mathrm{y}}_{1}+{\mathrm{y}}_{2}$ does not divide $P_\Delta$; this is made possible by the fact that each $x\in \mcX$ has at most $r-2$ outgoing arrows.

For $r=s+2$ and step $k\geq 3$, observe that $\g^{k-2}$ cannot be contained in the center. We distinguish two cases.

\emph{i)}
If  $\mb^{k-2}=\mz\oplus \Span{y_{1}}$ for some $y_{1}\in\mcX$ (so that in particular, $\g^{k-2}=\Span{y_{1}}\oplus \g^{k-1}$), we claim that there exists $y_{2}\in \mcX$ outside of $\mb^{k-2}$ that commutes with $y_{1}$ and such that $[y_{2},\g]\subset \mb^{k-2}$. For step three, $y_{2}$ exists because $\dim\ker\ad y_{1}\geq s+2$, and the second condition is automatic. For step $k\geq 4$, take $y_{2}$ in $\g^{k-3}\smallsetminus\mb^{k-2}$. Then~\eqref{eqn:ij} implies that $y_{1}, y_{2}$ commute. Choose $\sigma$  interchanging $y_{1}$ with $y_{2}$ and satisfying~\eqref{eqn:sigmamapsxitoyi}. The only arrows going out of $y_{1}$ end in the center, and the only arrows going out of $y_{2}$ end in either $y_{1}$ or the center, so $Q_\Delta$ has no invariant divisor.

\emph{ii)}
Suppose now that  $\mb^{k-2}=\mz\oplus \Span{y_{1},y_{2},\dots}$ for some $y_1,y_2\in \mcX$; then  $y_{1}$ and $y_{2}$ commute because of~\eqref{eqn:ij}, and the only arrows going out of  $y_{1}$, $y_{2}$ end in the center, so $Q_\Delta$ has no invariant divisor.
\smallskip

For $r=s+3$, we claim that there exist $y_1,y_2,y_3\in \mcX$ such that the sole linear factor of $P_\Delta$ in $\Z[{\mathrm{y}}_1,{\mathrm{y}}_2,{\mathrm{y}}_3]$ is ${\mathrm{y}}_2+{\mathrm{y}}_3$ and for every $x\in\mcB$, $1+{\mathrm{y}}_i-{\mathrm{x}}$ does not divide $Q_\Delta$.
In this case, an involution $\sigma$ such that $\sigma(y_{1})=y_{3}$, fixing $y_2$ and satisfying~\eqref{eqn:sigmamapsxitoyi} has the property that $P_\Delta Q_{\Delta}$ has no $\sigma$-invariant divisors.

To prove the claim, if $\mg$ is step $2$, each $x\in \mcX$ has at most $r-3$ outgoing arrows, so there exist $y_1, y_2,y_3\in \mcX$ such that $({\mathrm{y}}_{1}+{\mathrm{y}}_{2})({\mathrm{y}}_{1}+{\mathrm{y}}_{3})$ is coprime with $P_\Delta$.

Suppose $\mg$ is 3-step nilpotent; then
\[\mg\supset\mg^1\supset \mg^2\supset\mg^3=0.\]
Take $y_{1}\in\mcX$ such that $y_1\in \mg^1\smallsetminus \mg^2$; since $\ad_{y_{1}}:\mg\to \mz$ and $\dim\mz=s$, there exist distinct elements $ y_{2}, y_{3}\in \mcX$ commuting with $y_{1}$, so $({\mathrm{y}}_{2} +{\mathrm{y}}_{1})({\mathrm{y}}_{3}+{\mathrm{y}}_{1})$ does not divide $P_\Delta$. Moreover, we can assume that $1+{\mathrm{y}}_{2}-{\mathrm{y}}_{3}$ does not divide $Q_{\Delta}$ and ${\operatorname{Im}}(\ad_{y_{3}})\subset \Span{y_{1},y_{2}}$.  By construction, this basis verifies the claim.

Suppose the step $k$ is at least $4$ and let $y_{1}\in \mcX$ be such that $y_1\in \mg^{k-2}\smallsetminus \mz$. We split the proof into two cases.

\emph{i)} $\mb^{k-2}= \mz\oplus \Span{y_{1}}$ for some $y_1\in \mcX$. Pick $y_{2}\in \mcX$ such that $y_2\in \mg^{k-3}\smallsetminus \mb^{k-2}$; then as before, $y_2$ commutes with $y_1$. If $\dim \mb^{k-3}/ \mb^{k-2}\geq 2$, choose $y_{3}\in \mcX$ different from $y_{2}$ inside $\mg^{k-3}\smallsetminus \mb^{k-2}$. Otherwise,  $\mb^{k-3}= \mb^{k-2}\oplus\Span{y_{2}}$; if $k\geq 5$, pick $y_{3}\in \mcX$ inside $\mg^{k-4}\smallsetminus \mb^{k-3}$, which by~\eqref{eqn:ij} commutes with $y_{1}$. If $k=4$, let $y_{3}\in\mcX$ be an element commuting with $y_{1}$ inside the smallest possible ideal of the lower central series (indeed, $\mg^1$ or $\mg^0=\mg$). Then $[y_{3},\g]\subset \mz\oplus\Span{y_{1},y_{2}}$.

\emph{ii)} $\mb^{k-2}= \mz\oplus\Span{y_{1},y_{2},\dots}$ for some $y_1, y_2\in \mcX$. Take $y_3\in\mcX$, different from $y_{1},y_{2}$
; if possible take $y_3$ inside $\mg^{k-2}\smallsetminus \mz$, otherwise choose $y_{3}\in \mg^{k-3}\smallsetminus \mb^{k-2}$. In both cases, $[y_{3},\g]$ is contained in $\mz\oplus\Span{y_{1},y_{2}}$ and $[y_{2},\g]$ is contained in $\mz$; in addition, $({\mathrm{y}}_{1}+{\mathrm{y}}_{2})(\mathrm{y}_{1}+{\mathrm{y}}_{3})$ does not divide $P_\Delta$ because of~\eqref{eqn:ij}. Thus, there is a basis verifying our claim.
\end{proof}

It turns out that the bound $r\leq s+3$ is sharp. In order to demonstrate this, we will need the following observation:
\begin{lemma}\label{lm:sigz}
Let $\g$ be a nice Lie algebra  such that each node corresponding to an element of the center has at least $(r+s)/2$ incoming arrows; then any arrow-breaking involution maps elements of the center to elements outside the center.
\end{lemma}
\begin{proof}
For a contradiction, let $z$ be an element in the center such that $\sigma(z)$ is in the center; if $\sigma$ is arrow-breaking, the arrows ending at $z$ and $\sigma(z)$ have different labels, which is absurd.
\end{proof}

\begin{example}
An example of a nice nilpotent Lie algebra with $r=s+4$ that does not admit an arrow-breaking $\sigma$ is the two-step nilpotent Lie algebra with basis $\{e_1, \ldots, e_{16}\}$ and such that
\begin{gather*}
de^{11}=e^{15}+e^{24}+e^{39}+e^{6,10}+e^{78}\\
de^{12}=e^{16}+e^{2,10}+e^{35}+e^{48}+e^{79}\\
de^{13}=e^{17}+e^{29}+e^{36}+e^{45}+e^{8,10}\\
de^{14}=e^{18}+e^{27}+e^{34}+e^{56}+e^{9,10}\\
de^{15}=e^{19}+e^{28}+e^{3,10}+e^{46}+e^{57}\\
de^{16}=e^{1,10}+e^{23}+e^{47}+e^{59}+e^{68}
\end{gather*}
In this case, any $\sigma$ would have to map the center to elements of the complement of the center by Lemma~\ref{lm:sigz}. Thus, there is a set $\{y_1,y_2,y_3,y_4\}$ of elements outside the center which is invariant under $\sigma$. The Lie algebra has the property that $P_\Delta$ has at least $3$ linear factors involving only the variables ${\mathrm{y}}_1,\dotsc, {\mathrm{y}}_4$. Thus, $\sigma$ preserves a divisor of $P_\Delta$.

On the other hand, if we take the order two permutation $\sigma=(1\,5)(3\,11)\allowbreak (6\,12) \allowbreak (7\,13)\allowbreak (8\, 14)(9\, 15)(10\, 16)$, a direct computation using~\eqref{eqn:ricddadad} shows that the $\sigma$ diagonal me\-tric~\eqref{eqn:sigmametric} with $g_i=1$ for all $i$ is Ricci-flat. Notice that in this case not every metric of the form~\eqref{eqn:sigmametric} is Ricci-flat.
\end{example}

\begin{example}
In dimension $11$, an example of a two-step nice nilpotent Lie algebra  with  $r=s+5$ not admitting any arrow-breaking $\sigma$ is the Lie algebra with the form
\[
 (0,0,0,0,0,0,0,0,e^{12}+e^{34}+e^{56}+e^{78},e^{13}+e^{42}+e^{57}+e^{86}, e^{14}+e^{23}+e^{58}+e^{67})
\]
relative to a basis $\{e_1, \ldots, e_{11}\}$. Indeed, assume that $\sigma$ leaves no divisor of $P_\Delta Q_\Delta$ invariant. In particular $Q_\Delta$ has no invariant divisors, so $\sigma$ maps $\{e_9,e_{10},e_{11}\}$ to $\{e_1,\dotsc,e_8\}$. Thus, there is a fixed element in $\{e_1,\dotsc, e_8\}$; by symmetry, we can assume $\sigma (e_1)=e_1$. In order for $P_\Delta$ not to have invariant divisors, $\sigma$ must map $\{e_2,e_3,e_4\}$ to $\{e_5,\dotsc,e_{11}\}$. If all of $e_2,e_3,e_4$ are mapped into the center, then $\{e_5,e_6,e_7,e_8\}$ is invariant, and $P_\Delta$ has an invariant divisor. Thus, we can assume $\sigma(e_2)=e_5$. This implies that $\sigma(e_3),\sigma(e_4)\neq e_5,e_6,e_7,e_8$. So $e_3$ and $e_4$ are mapped into the center, as well as one of $e_6,e_7,e_8$. The elements in $\{e_6,e_7,e_8\}$ that are not mapped into the center by $\sigma$ determine an invariant factor of $P_\Delta$, giving a contradiction.

However, a $\sigma$-diagonal Ricci-flat metric can be constructed by taking
\[\sigma = (1\,9)(2\, 10)(3\, 11)(5\, 6)\]
and $\bil$ as in~\eqref{eqn:sigmametric} with $g_i=1$ for all $i$.
\end{example}

Proposition~\ref{prop:rs} applies to the class of two-step nilpotent Lie algebras attached to (undirected) graphs
introduced in~\cite{DaniMainkar:AnosovAutCompactNil}. Given a graph $(V,E)$, let $V_0$ be the free real vector space genereated by $V$ and let $V_1$ the subspace of $\Lambda^2V_0$ generated by $v\wedge v'$ where $v,v'$ are adjacent nodes in $(V,E)$. The attached Lie algebra is the vector space $V_0\oplus V_1$ where the nonzero Lie brackets are $[v,v']=v\wedge v'$.

\begin{corollary}
Any two-step nilpotent Lie algebra attached to a graph has an arrow-breaking involution.
\end{corollary}
\begin{proof}
Suppose the graph is connected. On a connected graph, the number of vertices $\abs{V}$ and the number of edges $\abs{E}$ are related by $\abs{E}\geq\abs{V}-1$. So, the attached Lie algebra has center of dimension $\abs{E}$ and dimension $\abs{E}+\abs{V}$; by Proposition~\ref{prop:rs}, it has an arrow-breaking involution.

If the graph is not connected, the attached Lie algebra is a direct sum of Lie algebras attached to its connected components.
\end{proof}

\section{Involutions on nilradicals of parabolic subalgebras}
\label{sec:parabolic}
In this section we recall a standard construction of nilpotent Lie algebras associated to a split simple Lie group and a subset $\Theta$ of the set of simple roots (see e.g.~\cite{Ko1}); for the simple Lie algebras $A_n,B_n,C_n$ and appropriate choices of $\Theta$, we obtain infinite families of Ricci-flat, nonflat nilpotent Lie algebras.

Let $\mgg$ be a split real simple Lie algebra with Iwasawa  decomposition $\mgg=\mk\oplus\ma\oplus\mn$ and root system $\Pi$. Let $\Sigma$ be a set of positive simple roots generating $\Pi$; we denote by $\Pi^+$ the set of positive roots.
As usual, if $\gamma\in\Pi^+$, then $x_\gamma$ denotes an arbitrary root vector in the one-dimensional root space $\mgg_\gamma$ and if $\alpha\in\Sigma$,  $\coord_{\alpha}(\gamma)$ denotes the $\alpha$-coordinate of $\gamma$
when it is expressed as a linear combination of simple roots. Let $\gamma_{\max}$ denote the unique maximal root of $\Pi^+$.

The set of parabolic Lie subalgebras of $\mgg$ containing the Borel subalgebra $\ma\oplus\mn$ is parametrized by
subsets of simple roots $\Sigma$ as follows.
Given a subset $\Theta\subset\Sigma$, denote $\lela\Theta\rira^\pm$ the set of positive/negative roots generated by $\Theta$. The corresponding parabolic subalgebra of $\mgg$
is $\mpp_\Theta$ where
\[
\mpp_\Theta=\ma\oplus \sum_{\gamma \in \Pi^+} \mg_\gamma \oplus \sum_{\gamma\in \lela\Theta\rira^-}\mg_\gamma=\ma \oplus \sum_{\gamma \in \lela\Theta\rira^+\cup \lela\Theta\rira^-} \mg_\gamma \oplus \sum_{\gamma\in \Pi^+\smallsetminus\lela\Theta\rira^+}\mg_\gamma.
\]
The nilradical of $\mpp_\Theta$ is the Lie algebra
\[
\mn_\Theta= \sum_{\gamma\in \Pi^+\smallsetminus\lela\Theta\rira^+}\mg_\gamma.
\]
This is a nilpotent Lie algebra and its lower central series (which coincides, after transposing the
indexes, with the upper central series) can be described as follows~\cite[Theorem 2.12]{Ko1}. Given $\gamma\in\Pi$, let
\[
o(\gamma)=\sum_{\alpha\in\Sigma\smallsetminus \Theta} \coord_{\alpha}(\gamma)
\]
be the {\em order} of $\gamma$ with respect to $\Theta$. The order can be positive, negative or zero. For any $\gamma\in \Pi^+$, $o(\gamma)\geq 0$ and $\gamma\in\Pi^+\smallsetminus \lela\Theta\rira^+$ if and only if $ o(\gamma)>0$.
For $i=0,1, \ldots $, let
\[
 \mg_{(i)}=\bigoplus_{\substack{
               \gamma\in\Pi^+ \\
               o(\gamma)=i
}} \mg_{\gamma}.
\]
If
$\mn_\Theta=\mn_\Theta^0\supset\mn_\Theta^1\supset\dots\supset \mn_\Theta^{k-1}\supset\mn_\Theta^k=0$ is
the lower central series of $\mn$, then
\begin{equation}\label{eq.lcs}
 \mn_\Theta^j=\bigoplus_{i=j+1}^{k}\mg_{(i)},\quad  k=o(\gamma_{\text{max}})\quad
 \text{ and }\quad \mn_\Theta^{k-1}=\mg_{(k)} \mbox{ is the center of }\mn.
\end{equation}
It follows from this description of the lower central series that
the nilradical $\mn_\Theta$ is abelian if and only if $\Sigma\smallsetminus\Theta=\{\alpha\}$ and $\coord_{\alpha}(\gamma_{\max})=1$.

\begin{proposition}\label{pro:nicebparab}
The set $\mathcal B=\{x_{\gamma}\st \gamma\in \Pi^+\smallsetminus \lela \Theta\rira^+\}$ is a nice basis of $\mn_\Theta$.
\end{proposition}
\begin{proof} The set $\mcB$ is clearly a basis of $\mnT$; moreover, the Lie bracket of $x_\gamma,x_\delta\in \mathcal B$ is given by
\begin{equation}\label{eq:bracket}
[x_\gamma,x_\delta]=
\begin{cases}
m_{\gamma,\delta}\, x_{\gamma+\delta}\neq 0&\mbox{if }\gamma+\delta\in \Pi^+\\
0&\mbox{if }\gamma+\delta\notin \Pi^+,
\end{cases}
\end{equation}
with $m_{\gamma,\delta}\in \R$,
so it is always a multiple of a single element in $\mathcal B$. Denote by $x^\gamma$ the elements in the dual basis of $\mcB$; then,
\[
dx^\gamma\in \Span{x^\delta\wedge x^\rho: \delta+\rho=\gamma}\subset \Lambda^2\mnT^*.
\]
This implies that for any $x_\delta$ in $\mcB$,
$x_\delta \hook \ dx^\gamma$ is either zero or a multiple of $x_{\gamma-\delta}$. Hence, $\mcB$ is a nice basis of $\mnT$.
\end{proof}

Let $\Delta$ denote the nice diagram associated to the nice basis $\mcB$ of $\mn_\Theta$ given by root vectors.
Order two permutations of $N(\Delta)$ are in one-to-one correspondence with order two permutations of $\Pi^+\smallsetminus \lela \Theta\rira^+$. Indeed, the nodes of $\Delta$ are the root vectors $x_\gamma$, so given a permutation $\sigma$ of $N(\Delta)$ and $\gamma, \delta\in \Pi^+\smallsetminus \lela\Theta\rira^+$, we set $\sigma(\gamma)=\delta$ if and only if $\sigma(x_\gamma)=x_\delta$. We will say that an order two permutation of $\gamma\in \Pi^+\smallsetminus \lela\Theta\rira^+$ is arrow-breaking when so is its corresponding permutation of $N(\Delta)$.

\begin{proposition}\label{pro:sigmaroot}
An order two permutation $\sigma$ of $\Pi^+\smallsetminus \lela \Theta\rira^+$ is arrow-breaking if and only if for any  $\gamma, \delta\in \Pi^+\smallsetminus \lela \Theta\rira^+$,
\begin{enumerate}
\item $\gamma+\delta\in \Pi^+\Rightarrow \sigma(\gamma)+\sigma(\delta)\notin\Pi^+$,
\item $\delta-\gamma\in \Pi^+\smallsetminus\lela \Theta\rira^+\Rightarrow \sigma(\delta)-\sigma(\gamma)\notin\Pi^+\smallsetminus\lela \Theta\rira^+$.
\end{enumerate}
\end{proposition}

\begin{proof}
The polynomials $P_\Delta$ and $Q_\Delta$ in~\eqref{eqn:PQ} can be easily described in terms of the  root system. Indeed, from~\eqref{eq.bracket} it is clear that
\begin{equation}
P_\Delta =
\prod_{
\substack{
\gamma,\delta\ \in\ \Pi^+\smallsetminus \lela\Theta\rira^+\\
\gamma+\delta\ \in\ \Pi^+
}}
({\mathrm{x}}_\gamma+{\mathrm{x}}_\delta),\qquad
Q_\Delta=\prod_{
\substack{
\gamma,\rho\ \in\ \Pi^+\smallsetminus \lela\Theta\rira^+\\
\rho-\gamma\ \in\ \Pi^+\smallsetminus \lela\Theta\rira^+
}} (1+{\mathrm{x}}_\gamma-{\mathrm{x}}_\rho).
\end{equation}
The result follows from this description of $P_\Delta$ and $Q_\Delta$ and Lemma~\ref{lemma:arrowbreaking}.
\end{proof}
According to Remark~\ref{remark:structureconstantsdonotmatter}, we do not make use of the actual structure constants of the parabolic nilradical to construct Ricci-flat metrics. Indeed, any arrow-breaking $\sigma$ as in Proposition~\ref{pro:sigmaroot} determines such a metric (or possibly a family, see Proposition~\ref{prop:family}) on every Lie algebra with the same nice diagram.
\begin{example}
Let $\mn_\emptyset$ be associated to $G_2$ with $\Theta=\emptyset$.  The system of positive roots is $\Pi^+=\{\alpha_1,\alpha_2,\alpha_1+\alpha_2,2\alpha_1+\alpha_2,3\alpha_1+\alpha_2,3\alpha_1+2\alpha_2\}$, with the maximal root appearing last. Following the above order of roots, we can chose a basis $\{e_1,\dots,e_6\}$ of $\mn_\emptyset$ such that each $e_i$ spans a root space $\g_\gamma$; this gives the nilpotent Lie algebra
\[
\texttt{64321:3}:\quad (0,0,-e^{12},e^{13},e^{14},e^{25}+e^{34}).
\]
There are exactly two arrow-breaking involutions of $\mn_\emptyset$, namely,  $\sigma$ which fixes $\alpha_2$ and $2\alpha_1+\alpha_2$ and satisfies
\[\sigma(\alpha_1)=3\alpha_1+2\alpha_2,\quad \sigma(\alpha_1+\alpha_2)=3\alpha_1+\alpha_2, \]
and $\tilde\sigma$ fixing $\alpha_1+\alpha_2$ and $3\alpha_1+\alpha_2$ and verifying
\[\tilde\sigma(\alpha_1)=3\alpha_1+2\alpha_2,\quad \tilde\sigma(\alpha_2)=2\alpha_1+\alpha_2. \]
Using the basis $\{e_1,\dots,e_6\}$, we can write $\sigma=(1\,6)(3\,5)$ and $\tilde{\sigma}=(1\,6)(2\,4)$.

The Ricci-flat metrics they induce are generically not flat. In fact, equation~\eqref{eqn:scnice} applied to a $\sigma$-diagonal metric $g$ gives:
\[
\lela R(e_1,e_5)e_1,e_5\rira=\frac{(g_1+g_3)^2}{4 g_2}
\]
showing that for $g_1\neq-g_3$ the metric is nonflat. For $g_1=-g_3$ an easy computation shows that the Riemann curvature is zero.

For a $\sigma$-diagonal metric $\tilde{g}$ induced by $\tilde{\sigma}$ we can use again formula~\eqref{eqn:scnice} applied to $e_1,e_5$ to compute directly:
\[
\lela\tilde{R}(e_1,e_5)e_1,e_5\rira=\frac{\tilde{g}_{1} \tilde{g}_{5}}{2 \tilde{g}_2}.\]
The parameters $\tilde{g}_{1}, \tilde{g}_{2},\tilde{g}_5$ are nonzero by definition, so these  metrics are not flat.
\end{example}

\subsection{Type A}
Consider the split Lie algebra of type $A$,  $\mgg=\mathfrak{sl}(n+1,\R)$ with $n\geq 1$. The set of positive roots is $\Pi^+=\{\eps_i-\eps_j\st 1\leq i<j\leq n+1\}$ which is spanned by the set of simple roots $\Sigma=\{\eps_i-\eps_{i+1}\st i=1, \ldots, n\}$. In fact,
\[
\eps_i-\eps_j= \sum_{l=i}^{j-1}\eps_l-\eps_{l+1}.
\]
Notice that if $i<j$ and $s<t$, $\eps_i-\eps_j+\eps_s-\eps_t\in \Pi^+$ if and only if $j=s$ or $i=t$.

For each $1\leq i <j\leq n+1$, the root space corresponding to $\eps_i-\eps_j$ is spanned by the matrix $E_{ij}\in \mgg$ with $ij$ entry equal to $1$ and all others equal to zero. The nilpotent Lie algebra associated to $\Theta=\emptyset$ is the Lie algebra of real strict upper triangular square matrices of size $n+1$.

Suppose $n=2k+1$ is odd and fix $\Theta=\{\alpha_{2i}\st i=1,\ldots,k\}\subset\Sigma$. Then $\lela \Theta\rira^+=\Theta$, since the roots in $\Theta$ are mutually orthogonal.
We will present a nice basis of $\mn_\Theta$ constituted by root vectors. The set $\Pi^+\smallsetminus \Theta$ of positive roots corresponding to $\mn_\Theta$ is
\begin{equation}
\label{eq:rootsa}
\{\eps_i-\eps_j\st 1\leq i, j\leq n+1, i+2\leq j \}\cup \{\eps_{2i-1}-\eps_{2i}\st i=1,\ldots,  k+1\}.
\end{equation}
For $i=1, \ldots,n+1$, set $h(i)=i+1$ if $i$ is odd and $h(i)=i+2$ otherwise;
$\mn_\Theta$ is the subalgebra of $(n+1)\times (n+1)$ upper triangular matrices such that for each row $i$, the possibly non-zero entries are in position $(i,j)$, with $j\geq h(i)$.
Notice that $\mn_\Theta$ is $k+1$-step nilpotent.
For instance, for $k=1$, $k=2$ the matrices in $\mn_\Theta$ have the following shapes:
\[
k=1 : \left(\begin{matrix}
0&\star &\star &\star \\
0&0 &0 &\star  \\
0&0 &0 &\star  \\
0&0 &0 &0
\end{matrix}\right)
\qquad
k=2 : \left(\begin{matrix}
0&\star &\star &\star &\star &\star \\
0&0 &0 &\star &\star &\star \\
0&0 &0 &\star &\star &\star \\
0&0 &0 &0 &0 &\star \\
0&0 &0 &0 &0 &\star \\
0&0 &0 &0 &0 &0
\end{matrix}\right)
\]

We will define an order two permutation on the set of positive roots corresponding to $\mn_\Theta$. For each $i=1, \ldots, n$ the set ${\operatorname{Ind}}_i=\{h(i),h(i)+1,\dots, n+1\}$ has an odd number of elements and ${\operatorname{Ind}}_{2l}={\operatorname{Ind}}_{2l+1}$ for $l=1, \ldots, k$.  Notice that $\eps_i-\eps_j\in \Pi^+\smallsetminus \Theta$ if and only if $j\in {\operatorname{Ind}}_i$.

Denote by ${\mathrm{s}}_i:{\operatorname{Ind}}_i\to {\operatorname{Ind}}_i$ the symmetry with respect to the mid-element; explicitly, ${\mathrm{s}}_i(j)=n+1+h(i)-j$. Define $\sigma$ on $\Pi^+\smallsetminus \Theta$ as follows.
\begin{equation}\label{eqn:sigmaA}
\sigma(\eps_i-\eps_j)=\begin{cases}
\eps_1-\eps_{{\mathrm{s}}_i(j)}&\mbox{ if }i=1,\\
\eps_{i-1}-\eps_{{\mathrm{s}}_i(j)} &\mbox{ if }i \mbox{ is even},\\
\eps_{i+1}-\eps_{{\mathrm{s}}_i(j)} &\mbox{ if }i>1 \mbox{ is odd}.
\end{cases}
\end{equation}
This permutation preserves the set $\{\eps_1-\eps_j\st j=1, \ldots, n+1\}$ and interchanges $\{\eps_{2l}-\eps_j\st j\in {\operatorname{Ind}}_{2l}\}$ with $\{\eps_{2l+1}-\eps_j \st j\in {\operatorname{Ind}}_{2l+1}\}$, for $l=1,\ldots, k+1$. Moreover, it reverses the natural order, in the sense that if $j<t$ belong to ${\operatorname{Ind}}_i$, then ${\mathrm{s}_i}(j)>{\mathrm{s}_i}(t)$.
The only root fixed by $\sigma$ is $\eps_1-\eps_{(n+1)/2}$.

\begin{proposition}
The permutation $\sigma$ defined by~\eqref{eqn:sigmaA} is arrow-breaking and therefore every $\sigma$-diagonal metric~\eqref{eqn:sigmametric} in $\mn_\Theta$ is Ricci-flat. These metrics are not flat.
\end{proposition}

\begin{proof}
We are going to show that $\sigma$ verifies the conditions in Proposition~\ref{pro:sigmaroot}.

First, consider $\gamma,\delta\in \Pi^+\smallsetminus \Theta$ such that $\gamma+\delta\in \Pi^+$. Without loss of generality, we may assume that $\gamma=\eps_i-\eps_j$, $\delta=\eps_j-\eps_t$. In particular, $j>1$ and
$\sigma(\delta)=\eps_{j\pm 1}-\eps_{{\mathrm{s}}_j(t)}$.  By construction $\sigma(\gamma)$ has the form $\eps_l-\eps_{{\mathrm{s}}_i(j)}$ for some $l\in\{i-1,i,i+1\}$. Then
\[
\sigma(\gamma)+\sigma(\delta)= \eps_l-\eps_{{\mathrm{s}}_i(j)}+\eps_{j\pm 1}-\eps_{{\mathrm{s}}_j(t)},
\]
and this is a root if and only if $l={\mathrm{s}}_j(t)$ or ${\mathrm{s}}_i(j)=j\pm1$. We know that ${\mathrm{s}}_i(j)=n+1+h(i)-j\neq j\pm1$ since $n$ is odd and $h(i)$ even, independently of $i$. Moreover, ${\mathrm{s}}_j(t)\in {\operatorname{Ind}}_j$ so ${\mathrm{s}}_j(t)\geq h(j)\geq j+1>i+1$ which implies ${\mathrm{s}}_j(t)\neq l$. Therefore, $\sigma(\gamma)+\sigma(\delta)$ is never a root.
\smallskip

Let $\gamma,\delta\in \Pi^+\smallsetminus \Theta$ be such that $\delta-\gamma\in \Pi^+\smallsetminus \Theta$; writing $\delta=\eps_i-\eps_t$, two situations may occur:
\begin{enumerate}
\item  $\gamma=\eps_i-\eps_j$ with $ h(i)\leq j <t$, or
\item   $\gamma=\eps_j-\eps_t$ with $h(i)\leq j$, $h(j)\leq t$.
\end{enumerate}

In the first case, independently of $i$ being odd or even,
\[
\sigma(\delta)-\sigma(\gamma)=\eps_{{\mathrm{s}}_i(j)}-\eps_{{\mathrm{s}}_i(t)},
\]
with ${\mathrm{s}}_i(t)<{\mathrm{s}}_i(j)$ since $j<t$. Hence, $\sigma(\delta)-\sigma(\gamma)$ is a negative root (and thus not in $\Pi^+$).

In the second case, $\sigma\delta=\eps_l-\eps_{{\mathrm{s}}_i(t)}$ and $\sigma\gamma=\eps_r-\eps_{{\mathrm{s}}_j(t)}$ for some $l$ and $r$ depending on $i$ and $j$, respectively. By construction $ {\mathrm{s}}_i(t) \neq {\mathrm{s}}_j(t) $, so
\[
\sigma(\delta)-\sigma(\gamma)=\eps_l-\eps_{{\mathrm{s}}_i(t)} - \eps_r+\eps_{{\mathrm{s}}_j(t)}
\]
is only a root when $l=r$, which by the definition of $\sigma$ implies $i=j$, against the hypothesis. Therefore, $\sigma(\delta)-\sigma(\gamma)$ is never a root.

We conclude that $\sigma$ verifies the conditions of Proposition~\ref{pro:sigmaroot} and is therefore arrow-breaking, inducing on $\mn_\Theta$ Ricci-flat metrics. These metrics are not flat because of criterion~\ref{enum:CriteriaFlat1} in Proposition~\ref{pr:CriteriaFlatness} applied to $e_s=x_{\eps_1-\eps_2}$, $e_t=x_{\eps_2-\eps_{\frac{n+1}2}}$, in which case $e_{k_1}=x_{\eps_1-\eps_{\frac{n+1}2}}$ is fixed by $\sigma$.
\end{proof}

\subsection{Type B}

The split real Lie algebra of type $B$ is $\mg=\so(n,n+1)$, $n\geq 2$, and has a system of positive roots given by
\[\{\eps_i\pm\eps_j\st 1\leq i<j\leq n \}\cup \{\eps_i \st 1\leq i \leq n\}.\] The simple roots in the system are $\eps_n$ and $\eps_i-\eps_{i+1}$ for $i=1, \ldots, n$.

For any $n\geq 3$, consider $\Theta=\{\eps_i-\eps_{i+1}\st 1\leq i\leq n-2\}$; the positive roots generated by this set are $\lela \Theta \rira^+=\{\eps_i-\eps_j\st 1\leq i<j\leq n-1\}$. Thus, $\mn_\Theta$ has a basis of root vectors $x_{\alpha}$ where $\alpha$ runs in the set of roots $\Pi^+\smallsetminus \lela \Theta\rira^+$, namely $\alpha$ is one of the following:
\begin{equation}
\label{eqn:rootB}
\begin{aligned}
\eps_i,\;  \eps_i-\eps_n, \; \eps_i+\eps_n, &\quad  1\leq i \leq n-1, \\
\eps_n,\;   \eps_i+\eps_j,&\quad 1\leq i <j\leq n-1.
\end{aligned}
\end{equation}

We shall construct an arrow-breaking involution $\sigma$ of $\Pi^+\smallsetminus \lela \Theta\rira^+$.
For $n=2$, $B_2$ is isomorphic to $C_2$ and this case will be treated in Example~\ref{ex:B2C2}, so we assume $n\geq 3$ for the rest of the section.

We claim that it is possible to define an arrow-breaking involution $\sigma$ satisfying
\begin{equation}
\label{eq:sigmaprop}
\left\{
\begin{array}{l}
\sigma (\eps_i+\eps_j)=\eps_h+\eps_k\Rightarrow\{i,j\}\cap \{h,k\}=\emptyset \mbox{ and }\\
\sigma (\eps_i+\eps_j)=\eps_h+\eps_n\Rightarrow i\neq h\neq j.
\end{array}\right.
\end{equation}
To this purpose, independently of $n$, we define
\[
\sigma(\eps_i-\eps_n)=\eps_i, \quad i=1, \ldots, n-1.
\]
For $n=3,4$ we set further
\begin{equation*}
n=3:\left\{
\begin{array}{l}
\sigma(\eps_1+\eps_2)=\eps_3,\\
\sigma (\eps_i+\eps_3)=\eps_i+\eps_3,
\end{array}\right.
\quad
n=4:\left\{
\begin{array}{l}
\sigma(\eps_1+\eps_2)=\eps_4,\\
\sigma (\eps_1+\eps_3)=\eps_2+\eps_4,\\
\sigma (\eps_2+\eps_3)=\eps_1+\eps_4,\\
\sigma (\eps_3+\eps_4)=\eps_3+\eps_4.
\end{array}\right.
\end{equation*}
Assume $n\geq 5$, choose $\sigma(\eps_i+\eps_n)=\eps_i+\eps_n$ for $i=4, \ldots, n-1$ and define $\sigma$ on
\[
\mathcal{S}:=\{\eps_i+\eps_j\st 1\leq i<j\leq n-1\}\cup\{\eps_n,\eps_1+\eps_n,\eps_2+\eps_n,\eps_3+\eps_n\}
\]
as follows.

Suppose that $n-1$ is even, then the indexes $i,j$ of the roots corresponding to the center of $\mn_\Theta$, namely  $\eps_i+\eps_j$ with $1\leq i<j\leq n-1$,  can be displayed in an upper triangular matrix as follows:
\begin{equation}\label{eqn:F1}
\begin{pmatrix}
\cline{3-6}\cline{8-9}
* & \rn{12}{_{1,2}} & \multicolumn{1}{|c}{\rn{13}{_{1,3}}} & \multicolumn{1}{c|}{\rn{14}{_{1,4}}} & \multicolumn{1}{|c}{\rn{15}{_{1,5}}} & \multicolumn{1}{c|}{\rn{16}{_{1,6}}}
& \ldots & \multicolumn{1}{|c}{\rn{1n2}{_{1,n-2}}}& \multicolumn{1}{c|}{\rn{1n1}{_{1,n-1}}}\\[5pt]
& * & \multicolumn{1}{|c}{\rn{23}{_{2,3}}} & \multicolumn{1}{c|}{\rn{24}{_{2,4}}} &  \multicolumn{1}{|c}{\rn{25}{_{2,5}}} & \multicolumn{1}{c|}{\rn{26}{_{2,6}}} & \ldots & \multicolumn{1}{|c}{\rn{2n2}{_{2,n-2}}}& \multicolumn{1}{c|}{\rn{2n1}{_{2,n-1}}} \\
\cline{3-6}\cline{8-9}
&  & * & \rn{34}{_{3,4}} & \multicolumn{1}{|c}{\rn{35}{_{3,5}}} & \multicolumn{1}{c|}{\rn{36}{_{3,6}}}&  & \multirow{2}{*}{\vdots} & \multirow{2}{*}{\vdots} \\[5pt]
& & & * & \multicolumn{1}{|c}{\rn{45}{_{4,5}}} & \multicolumn{1}{c|}{\rn{46}{_{4,6}}}\\
\cline{5-6}
& & & & \ddots & \ddots & \ddots & \vdots & \vdots\\
\cline{8-9}
& & & & & * & _{n-4,n-3} &
\multicolumn{1}{|c}{\rn{n4n2}{_{n-4,n-2}}} & \multicolumn{1}{c|}{\rn{n4n1}{_{n-4,n-1}}}\\[5pt]
& & & & & & * & \multicolumn{1}{|c}{\rn{n3n2}{_{n-3,n-2}}} & \multicolumn{1}{c|}{\rn{n3n1}{_{n-3,n-1}}}\\
\cline{8-9}
& & & & & & & * & _{n-2,n-1}\\
& & & & & & & & *
\end{pmatrix}
\begin{tikzpicture}[overlay,remember picture]
   \draw [stealth-stealth] (13) -- (24);
   \draw [stealth-stealth] (14) -- (23);
   \draw [stealth-stealth] (15) -- (26);
   \draw [stealth-stealth] (25) -- (16);
   \draw [stealth-stealth] (1n2) -- (2n1);
   \draw [stealth-stealth] (2n2) -- (1n1);
   \draw [stealth-stealth] (35) -- (46);
   \draw [stealth-stealth] (45) -- (36);
   \draw [stealth-stealth] (n4n2) -- (n3n1);
   \draw [stealth-stealth] (n3n2) -- (n4n1);
\end{tikzpicture}
\end{equation}

\noi Define $\sigma(\eps_i+\eps_j)=\eps_h+\eps_k$ if $i,j$ and $h,k$ are joined by an arrow in~\eqref{eqn:F1}, and set further $\sigma(\eps_1+\eps_2)=\eps_n$. It remains to define $\sigma$ on the $(n-3)/2$ elements of the form $\eps_{2l-1}+\eps_{2l}$ and $\eps_1+\eps_n$, $\eps_2+\eps_n$, $\eps_3+\eps_n$.

If $(n-3)/2$ is even, choose a $\sigma$ that fixes $\eps_i+\eps_n$ for $i=1, 2,3$ and acts on the elements $\eps_{2l-1}+\eps_{2l}$ without fixing any point. Otherwise, declare $\sigma$ to interchange $\eps_3+\eps_4$ with $\eps_1+\eps_n$, to fix $\eps_2+\eps_n$ and $\eps_2+\eps_n$ and to act on the remaining $\eps_{2l-1}+\eps_{2l}$ (an even number) without fixing any element.

Similarly, when $n-1$ is odd, the indexes of the roots $\eps_i+\eps_j$ with $1\leq i<j\leq n-1$ can be displayed as follows:
\begin{equation}\label{eqn:F2}
\begin{pmatrix}
\cline{2-5}\cline{7-8}
* & \multicolumn{1}{|c}{\rn{12}{_{1,2}}} & \multicolumn{1}{c|}{_{1,3}} & \multicolumn{1}{|c}{\rn{14}{_{1,4}}} & \multicolumn{1}{c|}{\rn{15}{_{1,5}}} & \ldots & \multicolumn{1}{|c}{\rn{1n2}{_{1,n-2}}}& \multicolumn{1}{c|}{\rn{1n1}{_{1,n-1}}}\\[5pt]
\cline{2-2}
& * & \multicolumn{1}{|c|}{\rn{23}{_{2,3}}} & \multicolumn{1}{|c}{\rn{24}{_{2,4}}} &  \multicolumn{1}{c|}{\rn{25}{_{2,5}}} &  \ldots & \multicolumn{1}{|c}{\rn{2n2}{_{2,n-2}}}& \multicolumn{1}{c|}{\rn{2n1}{_{2,n-1}}} \\
\cline{3-5}\cline{7-8}
&  & * & \multicolumn{1}{|c}{\rn{34}{_{3,4}}} & \multicolumn{1}{c|}{\rn{35}{_{3,5}}} &   & \multirow{2}{*}{\vdots} & \multirow{2}{*}{\vdots} \\[5pt]
\cline{4-4}
& & & * & \multicolumn{1}{|c|}{\rn{45}{_{4,5}}} & \\
\cline{5-5}
& & & & \ddots & \ddots  & \vdots & \vdots\\
\cline{7-8}
& & & & & * & \multicolumn{1}{|c}{_{n-3,n-2}} & \multicolumn{1}{c|}{_{n-3,n-1}}\\
\cline{7-7}
& & & & & & * & \multicolumn{1}{|c|}{_{n-2,n-1}}\\
\cline{8-8}
& & & & & & & *
\end{pmatrix}
\begin{tikzpicture}[overlay,remember picture]
   \draw [stealth-stealth] (14) -- (25);
   \draw [stealth-stealth] (24) -- (15);
   \draw [stealth-stealth] (1n2) -- (2n1);
   \draw [stealth-stealth] (1n1) -- (2n2);
\end{tikzpicture}
\end{equation}

\noi Again, we set $\sigma(\eps_i+\eps_j)=\eps_h+\eps_k$ whenever $i,j$ and $h,k$ are joined by an arrow in~\eqref{eqn:F2}. If the blocks with three elements are odd in number we define $\sigma(\eps_1+\eps_3)=\eps_n$, $\sigma (\eps_1+\eps_2)=\eps_3+\eps_n$, $\sigma (\eps_2+\eps_3)=\eps_1+\eps_n$, $\sigma (\eps_1+\eps_n)=\eps_1+\eps_n$, and interchange the corresponding elements of the remaining $3$-element blocks. For instance, we set $\sigma (\eps_3+\eps_4)=\eps_5+\eps_6$, $\sigma(\eps_3+\eps_5)=\eps_5+\eps_7$, $\sigma(\eps_4+\eps_5)=\eps_6+\eps_7$, and so on. If $3$-element blocks appear in an even number, then define $\sigma(\eps_1+\eps_2)=\eps_n$, $\sigma (\eps_3+\eps_4)=\eps_1+\eps_n$, $\sigma(\eps_i+\eps_n)=\eps_i+\eps_n$, $i=2,3$ and interchange with $\sigma$ the elements of the blocks as in the other case.

It is easy to check that in both cases $\sigma$ verifies~\eqref{eq:sigmaprop}.

\begin{proposition}
\label{prop:arrowbreakingBn}
The permutation $\sigma$ defined above is an arrow-breaking permutation and therefore every $\sigma$-diagonal metric~\eqref{eqn:sigmametric} in $\mn_\Theta$ is Ricci-flat. These metrics are not flat.
\end{proposition}

\begin{proof}
Let $\gamma,\delta$ be roots in $\Pi^+\smallsetminus \lela \Theta\rira^+$ such that $\gamma+\delta$ is a root. From~\eqref{eqn:rootB} one can describe all the possibilities for $\gamma$ and $\delta$ and it is not hard to check that $\sigma(\gamma)+\sigma(\delta)$ is never a root, because of how we constructed $\sigma$.

Let now $\gamma,\delta\in \Pi^+\smallsetminus \lela \Theta\rira^+$ be such that $\delta-\gamma$ is a root in  $\Pi^+\smallsetminus \lela \Theta\rira^+$. Again, from~\eqref{eqn:rootB} one can describe all possibilities for $\delta$ and $\gamma$. If $\delta=\eps_i$ for some $i=1, \ldots, n-1$, $\sigma(\delta)=\eps_i-\eps_n$, which cannot be written as the sum of two roots in $\Pi^+\smallsetminus \lela \Theta\rira^+$; in particular, $\sigma(\delta)-\sigma(\gamma)\notin\Pi^+\smallsetminus \lela \Theta\rira^+$.

If $\delta=\eps_i+\eps_n$ with $i\neq n$ then $\sigma(\delta)$ is either $\eps_i+\eps_n$ or $\eps_h+\eps_k$ with $1\leq h<k\leq n-1$. The roots $\gamma$ that we can subtract from $\delta$ are $\eps_i$ and $\eps_n$, whose images under $\sigma$ are $\eps_i-\eps_n$ and $\eps_1+\eps_l$ for some $l\neq n$, respectively. Therefore, $\sigma(\delta)-\sigma(\gamma)$ is not a root.

The last possibility for $\delta$ is $\delta=\eps_i+\eps_j$ with $1\leq i<j\leq n-1$. In this case $\sigma(\delta)$ is  either $\eps_n$, $\eps_h+\eps_k$ with  $1\leq h<k\leq n-1$, or $\eps_h+\eps_n$ with $1\leq  h\leq n-1$, and, in any case, $i\neq h,k\neq j$ because of~\eqref{eq:sigmaprop}.

For this $\delta$ we can choose three different roots for $\gamma$. First, $\gamma=\eps_i$ (or $\gamma=\eps_j$, which is analogous). We obtain $\sigma(\gamma)=\eps_i-\eps_n$ and $\sigma(\delta)-\sigma(\gamma)$ is not a root since $i\neq h,k\neq j$. Second, we might have $\gamma=\eps_i-\eps_n$  (or $\eps_j-\eps_n$) and then $\sigma(\gamma)=\eps_i$ which cannot be subtracted from $\sigma(\delta)$ because $i\neq h,k\neq j$. Finally, we can choose $\gamma=\eps_j+\eps_n$ being $\sigma(\gamma)$ either $\eps_j+\eps_n$ or $\eps_s+\eps_t$, but these cannot be subtracted from any of $\eps_n$, $\eps_h+\eps_k$ or $\eps_h+\eps_n$.

We have proved that $\sigma$ satisfies the conditions in Proposition~\ref{pro:sigmaroot}, so any $\sigma$-diagonal metric $\mn_\Theta$ on is Ricci-flat.

To show that these metrics are not flat for $n\geq 5$, take $e_s=x_{\eps_4}$ and $e_t=x_{\eps_n}$. Then
$x_{\eps_4+\eps_n}$ is fixed by $\sigma$ and the diagram contains the arrow
$x_{\eps_4}\xrightarrow{x_{\eps_n}}x_{\eps_4+\eps_n}$; thus, criterion~\ref{enum:CriteriaFlat1} in Proposition~\ref{pr:CriteriaFlatness} is satisfied and the metric is not flat.

For $n=3$, take $e_s=x_{\eps_1}$ and $e_t=x_{\eps_3}$; then, $[x_{\eps_1},x_{\eps_3}]=\lambda_1 x_{\eps_1+\eps_3}$ which is fixed  by $\sigma$, and $[x_{\eps_1}, x_{\eps_2}]=\lambda_2x_{\eps_1+\eps_2}$ where $\eps_2$ is not fixed by $\sigma$. Thus, equation~\eqref{eqn:scnice} implies $\lela R(e_s,e_t)e_s,e_t\rira\neq 0$ and the metric is not flat.

Finally, for $n=4$ take $e_s=x_{\eps_3}$ and $e_t=x_{\eps_4}$; then $[x_{\eps_3},x_{\eps_4}]=\lambda_1 x_{\eps_3+\eps_4}$, which is fixed  by $\sigma$. By criterion (C1) in Proposition~\ref{pr:CriteriaFlatness} we find that the metric is not flat.
\end{proof}

\begin{remark}
For large $n$, the Lie algebra $\mn_\Theta$ has a different Ricci-flat metric; indeed,
the center of the Lie algebra $\mn_\Theta$ is spanned by $x_{\eps_i-\eps_j}$, $1\leq i<j\leq n-1$ so it has dimension $(n-1)(n-2)/2$ and Proposition~\ref{prop:rs} applies. Notice that the resulting metrics are zero when restricted to the center, unlike the metrics constructed in Proposition~\ref{prop:arrowbreakingBn}.
\end{remark}

\subsection{Type C}
Consider the split Lie algebra of type $C$, namely $\g=\spg(2n,\R)$ for $n\geq 2$.
The set of positive roots is given by
\[\Pi^+=\{2\eps_i \st 1\leq i \leq n\}\cup\{\eps_i+\eps_j\st 1\leq i<j\leq n \}\cup \{\eps_i-\eps_j\st 1\leq i<j\leq n\},\]
generated by the simple roots $\Sigma=\{2\eps_n \}\cup \{\eps_i-\eps_{i+1}\st 1\leq i\leq n-1\}$.

The nonzero Lie brackets are the following:
\begin{align*}
{[x_{\eps_i-\eps_j},x_{2\eps_j}]}&=\lambda_{i,j}x_{\eps_i+\eps_j}\\
{[x_{\eps_i-\eps_j},x_{\eps_j+\eps_h}]}&=\mu_{i,j,h}x_{\eps_i+\eps_h}\\
{[x_{\eps_i-\eps_j},x_{\eps_k+\eps_j}]}&=\begin{cases}
\eta_{i,j,k}x_{\eps_i+\eps_k} & i<k \\
\eta'_{i,j,k}x_{\eps_k+\eps_i} & k<i
\end{cases}\\
{[x_{\eps_i-\eps_j},x_{\eps_j-\eps_h}]}&=\rho_{i,j,h}x_{\eps_i-\eps_h}
\end{align*}

\begin{example}\label{ex:B2C2}
Before going to the general case, we illustrate the nilradical associated to $C_2$, (which is isomorphic to $B_2$). The system of positive roots, given by $\Pi^+=\{\eps_1-\eps_2,\ \allowbreak2\eps_2,\ \eps_1+\eps_2,\ 2\eps_1\}$, is a nice basis representing the nice nilpotent Lie algebra \texttt{421:1} $(0,0,e^{12},e^{13})$. Note that the maximal root is $2\eps_1$.
We consider the involution $\sigma=(1\,4)(2\,3)$, i.e. $\sigma(\eps_1-\eps_2)=2\eps_1$ and $\sigma(\eps_1+\eps_2)=2\eps_2$. We see that $\sigma$ verifies the conditions of Proposition~\ref{pro:sigmaroot}, e.g.
\[
\eps_1-\eps_2 +2\eps_2=\eps_1+\eps_2\in \Pi^+,\qquad \eps_1-\eps_2+ \eps_1+\eps_2=2\eps_1\in \Pi^+,
\]
but
\[
\sigma(\eps_1-\eps_2)+\sigma(2\eps_2)=2\eps_1+\eps_1+\eps_2\notin \Pi^+,\quad  \sigma(\eps_1-\eps_2)+\sigma(\eps_1+\eps_2)=2\eps_1+2\eps_2\notin \Pi^+,
\]
and similarly for the other condition. We conclude that the corresponding $\sigma$-diagonal metrics are Ricci-flat. In fact, a direct computation applying formula~\eqref{eq:curv} shows that the Riemann tensor is zero. It is easy to see that the  arrow-breaking involution $\sigma$ is unique.
\end{example}

In the general case of $C_n$, for $n\geq 3$, consider $\Theta=\{\eps_i-\eps_{i+1}\st 1\leq i\leq n-2\}$. The positive roots generated by this set are $\lela \Theta \rira^+=\{\eps_i-\eps_j\st 1\leq i<j\leq n-1\}$.
We shall define a basis of root vectors in $\Pi^+\smallsetminus \lela \Theta\rira^+$ as follows:
\begin{align*}
\alpha_i&=\eps_i-\eps_n,\ i\neq n \qquad\qquad \beta_i=2\eps_i,\ i\neq n \\
\gamma_i&=\eps_i+\eps_n,\ i\neq 1 \qquad\qquad \delta_i=\eps_i+\eps_1,\ i\neq 1\\
\zeta_{i,j}&=\eps_i+\eps_j,\ i\neq j\text{ and } i,j\notin\{ 1,n\}.
\end{align*}
In particular $\beta_1=2\eps_1$, $\gamma_n=2\eps_n$ and $\delta_n=\eps_1+\eps_n$. We  get:
\begin{align*}
(\alpha_i+\gamma_j)&\in\Pi^+,\ i\neq n,j\neq 1\\
(\alpha_i+\delta_n)&\in\Pi^+,\ i\neq n.
\end{align*}
For the difference we have:
\begin{align*}
(\zeta_{i,j}-\alpha_i),(\zeta_{i,j}-\gamma_j)&\in\Pi^+\smallsetminus \lela \Theta\rira^+, i\neq j\text{ and }
 i,j\notin\{ 1,n\}\\
(\gamma_i-\alpha_i),(\gamma_i-\gamma_n)&\in\Pi^+\smallsetminus \lela \Theta\rira^+,i\neq 1\\
(\delta_j-\alpha_1),(\delta_j-\gamma_j)&\in\Pi^+\smallsetminus \lela \Theta\rira^+,j\neq 1\\
(\beta_i-\alpha_i),(\beta_i-\gamma_i)&\in\Pi^+\smallsetminus \lela \Theta\rira^+,i\neq 1,n\\
(\delta_i-\alpha_i),(\delta_i-\delta_n)&\in\Pi^+\smallsetminus \lela \Theta\rira^+,i\neq 1,n\\
(\beta_1-\alpha_1),(\beta_1-\delta_n)&\in\Pi^+\smallsetminus \lela \Theta\rira^+.
\end{align*}
We can construct an arrow-breaking involution $\sigma$ such that
\[\sigma(\alpha_i)=\beta_{n-i},\ i\neq n \qquad \sigma(\gamma_j)= \delta_{n-j+2},\ j\neq 1\]
and the other elements are fixed. It is easy to see that this involution satisfies the conditions of Proposition~\ref{pro:sigmaroot}, hence it is arrow-breaking and by Proposition~\ref{prop:2OrderSigmaEnhancedGiveRicciFlat} the associated $\sigma$-diagonal metrics are Ricci-flat.

For $n=3$, this construction yields the $8$-dimensional Lie algebra \[(0,0,0,e^{12},e^{23},e^{14},e^{15}+e^{34},e^{35})\] with $\sigma= (1\,8)(2\,7)(3\,6)(4\,5)$; in this case the $\sigma$-diagonal metric turns out to be flat.

For $n\geq 4$ the metrics are Ricci-flat but not flat; if $n>4$, we can apply  criterion~\ref{enum:CriteriaFlat1} of Proposition~\ref{pr:CriteriaFlatness} to $\alpha_2, \gamma_3$: indeed the bracket $[x_{\alpha_2},x_{\gamma_3}]=[x_{\eps_2-\eps_n},x_{\eps_3+\eps_n}]=x_{\eps_2+\eps_3}=x_{\zeta_{2,3}}$ is a fixed point of $\sigma$, and  there are no arrows of type $x_{\alpha_2}\xrightarrow{}x_{\delta_{n-1}}=\sigma(x_{\gamma_3})$ or $x_{\gamma_3}\xrightarrow{}x_{\beta_{n-2}}=\sigma(x_{\alpha_2})$. For $n=4$, we can choose $\gamma_4$ instead of $\gamma_3$ and again apply~\ref{enum:CriteriaFlat1}.

\section{Maximal nice Lie algebras and low-di\-men\-sio\-nal Ricci-flat metrics}
\label{sec:maximal}
Given two nice diagrams with $n$ nodes $\Delta,\Delta'$, we will write $\Delta\leq\Delta'$ if there is a bijection from $N(\Delta)$ to $N(\Delta')$ mapping $\mathcal{I}_\Delta$ to a subset of $\mathcal{I}_{\Delta'}$; if $\g$, $\g'$ are Lie algebras with diagrams $\Delta, \Delta'$, we will also write  $\g\leq\g'$.

Notice that $\g\leq \g'$ and $\g'\leq \g$ can both be true for nonisomorphic $\g,\g'$ with isomorphic diagrams; however, $\leq$ defines a partial order relation on isomorphism classes of nice diagrams.

This partial order is relevant for the construction of Ricci-flat metrics because of the following:
\begin{lemma}
\label{lemma:order}
Let  $\g,\g'$ be nice Lie algebras with nice diagrams $\Delta\leq \Delta'$. If $\g'$ has an arrow-breaking involution, then $\g$ also has an arrow-breaking involution.
\end{lemma}
\begin{proof}
Since $\Delta\leq\Delta'$, we can assume $\g$ and $\g'$ have the same set of nodes and that arrows of $\Delta$ are also arrows of $\Delta'$. Thus, an arrow-breaking involution for $\Delta'$ is also an arrow-breaking involution for $\Delta$.
\end{proof}

Lemma~\ref{lemma:order} effectively reduces the problem to the smaller class of maximal nice Lie algebras; we will say that a nice nilpotent Lie algebra is \emph{maximal} if its diagram is maximal in the class of isomorphism classes of nice diagrams associated to a nilpotent Lie algebra.

\begin{example}
The nice Lie algebra $(0,0,e^{12})$ is maximal, because adding an arrow would either create a cycle or multiple arrows with same source and destination, breaking the nice diagram condition.
\end{example}

In the definition of maximality, one only considers nice diagrams associated to a Lie algebra. This restriction is made necessary by the fact that a  maximal nice nilpotent Lie algebra may have a nonmaximal diagram, in the sense that it is possible to add arrows while retaining the conditions defining a nice diagram.

For instance, the nice nilpotent Lie algebras
\[
\texttt{85421:4}\quad(0,0,0,e^{12},\pm e^{14},e^{13}+e^{24},e^{15},e^{17}+e^{23}),
\]
are maximal. In this case, it is possible to add arrows $e_2\xrightarrow{e_6}e_7$, $e_3\xrightarrow{e_4}e_7$, $e_4\xrightarrow{e_6}e_8$ to its diagram by preserving the nice condition, but any such diagram will not have any associated Lie algebras.

\begin{theorem}
The list of maximal nice nilpotent Lie algebras in dimension $\leq7$ is given in Table~\ref{table:maximal}.
\end{theorem}
\begin{proof}
Going through the classification of~\cite{ContiRossi:Construction}, one sees that for each nice Lie algebra of dimension $\leq7$ not appearing in Table~\ref{table:maximal} it is always possible to add a bracket so as to obtain a nice Lie algebra; for instance (omitting the obvious case of abelian Lie algebras), in dimension $\leq 5$ we find
\begin{align*}
(0,0,0,e^{12})&\leq (0,0,e^{14},e^{12})\\ 
(0,0,e^{12},e^{13},e^{14})&\leq (0,0,e^{12},e^{13},e^{14}+e^{23})\\
(0,0,e^{12},e^{13},e^{23})&\leq (0,0,e^{12},e^{13},e^{14}+e^{23})\\
(0,0,0,e^{12},e^{14})&\leq (0,0,0,e^{12},e^{14}+e^{23})\\
(0,0,0,e^{12},e^{24}+e^{13})&\leq (0,0,e^{14},e^{12},e^{24}+e^{13})\\
(0,0,0,e^{12},e^{13})&\leq (0,0,0,e^{12},e^{13}+e^{24})\\
(0,0,0,0,e^{12})&\leq (0,0,0,0,e^{12}+e^{34})\\
(0,0,0,0,e^{34}+e^{12})&\leq (0,0,0,e^{13},e^{34}+e^{12}).
\end{align*}
It follows that none of the nice nilpotent Lie algebras on the left side is maximal. The same argument can be used in dimensions $6,7$ to prove that the nice nilpotent Lie algebras not appearing in Table~\ref{table:maximal} are not maximal.

It remains to prove that the Lie algebras appearing in Table~\ref{table:maximal} are maximal. For dimension $\leq5$ this is by exclusion, since at least one maximal nice nilpotent Lie algebras must exist in any dimension.

For dimensions $6,7$, observe that all nice Lie algebras appearing in Table~\ref{table:maximal} satisfy  $i<j$ whenever $e_i\to e_j$ is an arrow. For each diagram $\Delta$ with nodes $\{1,\dotsc, n\}$ satisfying this condition, define the set
\[S(\Delta)=\{\rho\in\Sigma_n\st \rho_i<\rho_j \text{ whenever } e_i\to e_j\};\]
if for every $\rho\in S(\Delta)$ it is not possible to add a pair of arrows $e_i\xrightarrow{e_j}e_k$, $e_j\xrightarrow{e_i}e_k$ with $\rho_i,\rho_j<\rho_k$ in such a way that the resulting diagram satisfies
\ref{enum:condNice1}--\ref{enum:condNice3}, it follows that $\Delta$ is maximal.

We now proceed to prove that the following Lie algebras are maximal:
\begin{align*}
\texttt{64321:2} &\quad (0,0,e^{12},e^{13},e^{14},e^{15}+e^{23})\\
\texttt{64321:4} &\quad (0,0,e^{12},e^{13},e^{14}+e^{23},e^{15}+e^{24})\\
\texttt{64321:5} &\quad (0,0, -e^{12},e^{13},e^{14}+e^{23},e^{25}+e^{34})\\
\texttt{6431:2a} &\quad (0,0,e^{12},e^{13},e^{23},e^{14}+e^{25})\\
\texttt{6431:2b} &\quad (0,0,e^{12},e^{13},e^{23},e^{14}-e^{25})\\
\texttt{632:3a} &\quad (0,0,0,e^{12},e^{14}+e^{23},e^{13}+e^{24})\\
\texttt{632:3b} &\quad (0,0,0,e^{12},e^{14}+e^{23},e^{13}-e^{24}).
\end{align*}

In order to prove that $\texttt{64321:2}=(0,0,e^{12},e^{13},e^{14},e^{15}+e^{23})$ is maximal, observe that in this case,  $S(\Delta)$ contains the identity and $(1\,2)$. Thus by the ordered condition and~\ref{enum:condNice1}, the only arrows that can be added are
\[e_2\xrightarrow{e_4}e_5,\ e_2\xrightarrow{e_4}e_6,\ e_2\xrightarrow{e_5} e_6,\ e_3\xrightarrow{e_4}e_5,\ e_3\xrightarrow{e_4}e_6,\ e_3\xrightarrow{e_5}e_6,\ e_4\xrightarrow{e_5}e_6.\]
Each choice violates~\ref{enum:condNice2}.

For the others, it is easy to check that arrows of the form $e_i\xrightarrow{e_j}e_k$ with $i,j<k$ cannot be added preserving the nice condition; in each case we must consider the arrows that are not of this type, but satisfy $\rho_i,\rho_j<\rho_k$ for some $\rho\in S(\Delta)$.

For $\texttt{64321:4}=(0,0,e^{12},e^{13},e^{14}+e^{23},e^{15}+e^{24})$ and $\texttt{64321:5}=(0,0, e^{12},\allowbreak e^{13},e^{14}+e^{23},e^{25}-e^{34})$, $S(\Delta)$ is generated by $(1\,2)$, so there is no additional arrow satisfying~\ref{enum:condNice1} to consider.

For $\texttt{6431:2}=(0,0,e^{12},e^{13},e^{23},e^{14}\pm e^{25})$,  $S(\Delta)$ is the group generated by $(1\,2)$ and $(4\,5)$, so we must additionally consider the arrows $e_5\xrightarrow{e_4}e_6$, $e_1\xrightarrow{e_5}e_4$, $e_3\xrightarrow{e_5}e_4$; each of them violates~\ref{enum:condNice2}.

For $\texttt{632:3}=(0,0,0,e^{12},e^{14}+e^{23},e^{13}\pm e^{24})$, $S(\Delta)$ contains the group generated by $(1\,2)$, $(1\,3)$ and $(5\,6)$, and additionally the elements  $(3\,4)$, $(3\,4)(5\,6)$, $(1\,2)(3\,4)$, $
(1\,2)(3\,4)(5\,6)$. So we must additionally consider the arrows $e_i\xrightarrow{e_6}e_5$ for $i=1,\dotsc, 4$, and
each of them violates~\ref{enum:condNice2}.

A similar argument proves the maximality of the $7$-dimensional Lie algebras in the list.
\end{proof}

{\setlength{\tabcolsep}{2pt}
\begin{table}[thp]
\centering
{\caption{\label{table:maximal} Maximal nice Lie algebras $\g$ with $3\leq n\leq 7$ nodes and, for each except {\ttfamily 64321:5}, an arrow-breaking involution $\sigma$, and a plane giving a nonzero component of  the curvature tensor.}
\begin{small}
\begin{tabular}{>{\ttfamily}l L L L}
\toprule
\textnormal{Name} & \g & \sigma & \text{plane}\\
\midrule
31:1&0,0,e^{12}& (1\,3)\\[3pt]
421:1&0,0,e^{12},e^{13}&(1\,4)(2\,3)\\[3pt]
5321:2&0,0,e^{12},e^{13},e^{14}+e^{23}& (2\,4)(1\,5)&e_1\wedge e_2\\[3pt]
64321:2&0,0,e^{12},e^{13},e^{14},e^{15}+e^{23}&
(2\, 4)(1\, 6)& e_1\wedge e_2\\
64321:4&0,0,e^{12},e^{13},e^{14}+e^{23},e^{15}+e^{24}&
(3\, 4)(2\, 5)(1\, 6)\\
64321:5&0,0,-e^{12},e^{13},e^{14}+e^{23},e^{25}+e^{34}&
\text{none}\\
6431:2&0,0,e^{12},e^{13},e^{23},e^{14}\pm e^{25}&
(3\, 4)(2\, 6)(1\, 5)\\
632:3&0,0,0,e^{12},e^{14}+e^{23},e^{13}\pm e^{24}&
(2\, 6)(1\, 5)&e_1\wedge e_2\\[3pt]
754321:2&0,0,e^{12},e^{13},e^{14},e^{15},e^{16}+e^{23}&(3\, 5)(1\, 7)&e_1\wedge e_3\\
754321:3&0,0,e^{12},e^{13},e^{14},e^{15}+e^{23},e^{16}+e^{24}&(3\, 4)(2\, 6)(1\, 7)&e_1\wedge(e_2-e_4)\\
\multirow{2}{*}{754321:9}&
\multicolumn{1}{L}{0,0,{(1-\lambda)} e^{12},e^{13}, \lambda e^{14}+e^{23},}&\multirow{2}{*}{$(3\, 5)(2\, 6)(1\, 7)$}\\*
&\multicolumn{1}{R}{e^{24}+e^{15},e^{25}+e^{34}+e^{16}}&\\
75432:2&0,0,-e^{12},e^{13},e^{14},e^{15}+e^{23},e^{25}+e^{34}&(3\, 5)(2\, 7)(1\, 6)&(e_1-e_2)\wedge e_3\\
75421:2&0,0,e^{12},e^{13},e^{23},e^{14},e^{16}+e^{25}&(3\, 4)(2\, 7)(1\, 5)&e_1\wedge(e_2-e_4)\\
75421:5&0,0,-e^{12},\pm e^{13},e^{23},e^{14}+e^{25},e^{26}+e^{34}&(3\, 5)(2\, 7)(1\, 6)&(e_1-e_2)\wedge e_3\\
7542:3&0,0,e^{12},e^{13},e^{23},e^{15}+e^{24},e^{14}\pm e^{25}&(3\, 5)(2\, 7)(1\, 6)&e_1\wedge(e_2-e_3)\\
74321:10&0,0,0,-e^{12},e^{14},e^{15}+e^{23},e^{26}+e^{13}+e^{45}&(3\, 5)(2\, 7)(1\, 6)&(e_1-e_5)\wedge(e_2+e_3)\\
7431:4&0,0,0,e^{12},e^{14},e^{13}+e^{24},e^{15}+e^{23}&(3\, 7)(2\, 5)(1\, 6)&e_1\wedge e_2\\
7431:11&0,0,0,e^{12},\pm e^{13}+e^{24},e^{14}+e^{23},e^{15}+e^{26}&(3\, 5)(2\, 7)(1\, 6)&(e_1-e_3)\wedge e_2\\
\multirow{2}{*}{7431:13}&\multicolumn{1}{L}{0,0,0,(A-1)e^{12},\pm e^{14}+ e^{23},}&\multirow{2}{*}{$(3\, 5)(2\, 7)(1\, 6)$}&\multirow{2}{*}{$(e_1+e_2)\wedge (e_2+e_3)$}\\
&\multicolumn{1}{R}{Ae^{13}+e^{24}, e^{15}+e^{26}+e^{34}}&\\
7421:11&0,0,0,e^{12},e^{13},e^{14},e^{16}+e^{24}\pm e^{35}&(3\, 4)(1\, 7)&e_1\wedge e_3\\
7421:13&0,0,0,e^{12},e^{13},e^{24},e^{14}+e^{26}+e^{35}&(3\, 6)(2\, 7)(1\, 5)&(e_1-e_3)\wedge e_2\\
742:13&0,0,0,e^{12},e^{13},e^{14}+e^{23},e^{15}+e^{24}&(3\, 7)(2\, 5)(1\, 6)&(e_1-e_2)\wedge e_3\\
742:14&0,0,0,\pm e^{12},e^{13},e^{14}+e^{23},e^{24}+ e^{35}&(3\, 7)(2\, 6)(1\, 5)&(e_1-e_3)\wedge e_2\\
742:16&0,0,0,e^{12},e^{13},e^{15}+e^{24},e^{14}+e^{35}&(5\, 7)(2\, 3)(1\, 6)& e_1\wedge e_2\\
741:5&0,0,0,e^{12},e^{13},e^{23},e^{15}+e^{24}+e^{36}&(3\, 7)(2\, 6)(1\, 4)& e_1\wedge e_3\\
73:7&0,0,0,0,e^{14}+e^{23},e^{13}\pm e^{24},e^{12}+e^{34}&(4\, 6)(2\, 7)(1\, 5)& (e_1+e_2)\wedge (e_2+e_4)\\
\bottomrule
\end{tabular}
\end{small}
}
\end{table}
}

\begin{proposition}\label{pro:SigmaEnhancedUpDim7}
Every nice nilpotent Lie algebra of dimension $\leq 7 $ has an arrow-breaking involution except 
\[\emph{\texttt{64321:5}}\quad (0,0,-e^{12},e^{13},e^{14}+e^{23},e^{25}+e^{34})\]
\end{proposition}
\begin{proof}
We first prove that the Lie algebra in the statement has no arrow-breaking involution $\sigma$.

Suppose such a $\sigma$ exists. Since $Q_\Delta$ and $\sigma Q_\Delta$ are coprime and $(1+{\mathrm{e}}_1-{\mathrm{e}}_5)(1+{\mathrm{e}}_2-{\mathrm{e}}_5)(1+{\mathrm{e}}_3-{\mathrm{e}}_5)(1+{\mathrm{e}}_4-{\mathrm{e}}_5)$ divides $Q_\Delta$, it follows that  $\sigma Q_\Delta$ does not have any linear factor of the form $(1+{\mathrm{e}}_i-{\mathrm{e}}_5)$. Thus, $\sigma_5$ is either $1$ or $2$. The same argument applies to $6$. It follows that ${\mathrm{e}}_3+{\mathrm{e}}_4$ is $\sigma$-invariant; since it divides $P_\Delta$, we reach a contradiction.

Table~\ref{table:maximal} gives a list with one arrow-breaking involution for each of the other  nice nilpotent Lie algebras of dimension $\leq 7$. We point out that the involution is generally not unique.
\end{proof}

\begin{example}\label{ex:missing6NiceRicciFlat} 
Consider the nice Lie algebra
\[\texttt{64321:5}\quad(0,0,- e^{12},e^{13},e^{14}+e^{23},e^{25}+e^{34}),\]
that does not admit any arrow-breaking involution. We will show that for some $\sigma$-diagonal metrics the Ricci tensor can be zero. Using the involution $\sigma=(1\,3)(4\,5)$, the Ricci tensor associated to the $\sigma$-diagonal metrics \eqref{eqn:sigmametric} is:
\[\ric=
 \frac{(g_{1}+g_{2}) g_{4}}{2 g_{1} g_{2}}e^1\odot e^4 +
  \left(\frac{g_{4}}{2 g_{1}}-\frac{g_{6}}{2 g_{4}} \right)e^2\odot e^3
   -\left(\frac{g_{4}^2}{2 g_{1}^2}+\frac{g_{6}}{2 g_{2}}\right)
  e^5\otimes e^5.
  \]
  where $e^1\odot e^j=e^i\otimes e^j+e^j\otimes e^i$.
Clearly, the Ricci tensor
is zero if $g_{2}=-g_{1}$ and $g_{6}=g_{4}^2/g_{1}$. Observe that  $e_2\xrightarrow{e_5}e_6=\sigma(e_6)$, and there are no arrows of the form $e_2\rightarrow e_{\sigma_5}$, $e_5\rightarrow e_{\sigma_2}$, so by the Criteria~\ref{enum:CriteriaFlat1} of Proposition~\ref{pr:CriteriaFlatness}, we can conclude that the metric is not flat.

Moreover, the following $\sigma$ can define $\sigma$-diagonal Ricci-flat non-flat metric:
\begin{align*}
(1\,6)(3\,5)&\quad\text{ for }g_1=g_3\\
(2\,4)(3\,5)(1\,6)&\quad\text{ for }g_1=\pm g_3\\
(2\,5)(1\,6)&\quad\text{ for }g_3=\frac{g_2^2 (g_1^2-g_4^2)}{g_1^2 g_4}.
\end{align*}
\end{example}

From Proposition~\ref{pro:SigmaEnhancedUpDim7} and the previous example we obtain:
\begin{theorem}\label{teo:nicedimless7}
Every nice nilpotent Lie algebra of dimension $\leq 7 $ has a Ricci-flat metric.
\end{theorem}
It is natural to ask if the metric can be chosen to be nonflat. In dimension $3$, this is obviously not possible, since the Ricci tensor determines the full curvature tensor; moreover, it is clear that abelian Lie algebras are necessarily flat. In dimension $4$, every Ricci-flat metric on the Lie algebra $(0,0,0,12)$ is flat \cite{BSV15,Su16}. 

With these exceptions, we can indeed show that the metric can be chosen to be nonflat:
\begin{corollary}
Every nonabelian nice nilpotent Lie algebra of dimension $\leq 7 $ not isomorphic to $(0,0,12)$ or $(0,0,0,12)$ has a nonflat Ricci-flat metric.
\end{corollary}
\begin{proof}
For $(0,0,12,13)$, it is easy to check that the only arrow-breaking order two permutation is $(1\,4)(2\,3)$, and every diagonal metric~\eqref{eqn:sigmametric} induced by it is flat. Nevertheless, one can prove that the $\sigma$-diagonal metric induced by $\sigma=(2\,4)$, with parameters $g_1=-1$, $g_2=g_3=1$ is a Ricci-flat nonflat metric.

For $(0,0,0,0,12)$, we see that $\sigma=(2\,4)(1\,3)$ is arrow-breaking and that $\sigma$-diagonal metrics are not flat (for instance, apply \ref{enum:CriteriaFlat1} of Proposition~\ref{pr:CriteriaFlatness}).

For the other cases, we observe first that if $\g$ has a nonflat Ricci-flat metric, then so has $\g\times\R^k$; therefore, we only need to consider nice Lie algebras whose diagram does not have a disconnected node.

For $19$ maximal Lie algebras and $1$ family, we can find a plane $v\wedge w$ such that, given $\sigma$ as in
Table~\ref{table:maximal}, a generic $\sigma$-compatible metric \eqref{eqn:sigmametric} satisfies $\lela R(v,w)v,w\rira\neq 0$, not only on the maximal Lie algebra $\g$, but also for all nice Lie algebras $\g'\leq \g$ whose diagram does not have a disconnected node. Such a plane is indicated in the last column of Table~\ref{table:maximal}.

In dimension $5\leq n\leq 7$, there are exactly $17$ nonabelian nice nilpotent Lie algebras and $1$ family 
whose diagrams have no disconnected nodes and cannot be written as $\g'\leq \g$, with $\g$ one of the 19 Lie algebras or the family above. One is \texttt{64321:5}, which we have already shown to admit a nonflat Ricci-flat metric. The others are listed in Table~\ref{table:missing} together with an arrow-breaking $\sigma$ and a plane such that the generic $\sigma$-diagonal metric is nonflat.
\end{proof}

{\setlength{\tabcolsep}{2pt}
\begin{table}[thp]
\centering{
\caption{\label{table:missing} Nice nilpotent Lie algebras for which existence of a nonflat Ricci-flat metric does not follow from Table~\ref{table:maximal}.}
\begin{small}
\begin{tabular}{>{\ttfamily}l L L L}
\toprule
\textnormal{Name} & \g & \sigma & \text{plane}\\
\midrule
64321:3&0,0,- e^{12},e^{13},e^{14},e^{25}+e^{34}&(2\, 4)(1\, 6)&e_1\wedge e_2\\
64321:4&0,0,e^{12},e^{13},e^{14}+e^{23},e^{15}+e^{24}&(3\, 4)(2\, 5)(1\, 6)&e_1\wedge(e_2+e_3)\\
6431:2&0,0,e^{12},\pm e^{13},e^{23},e^{14}+e^{25}&(3\, 4)(2\, 6)(1\, 5)&e_1\wedge e_2\\
6431:3&0,0,e^{12},e^{13},e^{23},e^{15}+e^{24}&(3\, 5)(2\, 6)(1\, 4)&(e_1-e_3)\wedge e_2\\
6321:4&0,0,0,- e^{12},e^{14}+e^{23},e^{15}+e^{34}&(3\, 5)(2\, 4)(1\, 6)&e_1\wedge (e_2+e_3)\\
631:5&0,0,0,\pm e^{12},e^{13},e^{24}+e^{35}&(2\, 3)(1\, 6)&e_1\wedge e_2\\
631:6&0,0,0,e^{12},e^{13},e^{25}+e^{34}&(2\, 3)(1\, 6)&e_1\wedge e_2\\
754321:5&0,0,- e^{12},e^{13},e^{14},e^{15},e^{16}+e^{25}+e^{34}&(3\, 5)(1\, 7)&e_1\wedge e_3\\
754321:6&0,0,e^{12},e^{13},e^{14}+e^{23},e^{15}+e^{24},e^{16}+e^{34}&(4\, 6)(2\, 5)(1\, 7)&e_1\wedge e_2\\
754321:7&0,0,e^{12},e^{13},e^{14}+e^{23},e^{15}+e^{24},e^{16}+e^{25}&(3\, 5)(2\, 6)(1\, 7)& e_1\wedge e_3\\
\multirow{2}{*}{754321:9}&\multicolumn{1}{L}{0,0,(1-\lambda) e^{12} ,e^{13},\lambda e^{14}+e^{23},}&\multirow{2}{*}{$(3\, 5)(2\, 6)(1\, 7)$}& \multirow{2}{*}{$e_1\wedge e_3$}\\
&\multicolumn{1}{R}{e^{15}+e^{24},e^{16}+e^{34}+e^{25}} & & \\
75432:3&0,0,- e^{12},e^{13},e^{14}+e^{23},e^{15}+e^{24},e^{25}+e^{34}&(3\, 5)(2\, 7)(1\, 6)&e_1\wedge e_2\\
75421:4&0,0,e^{12},e^{13},e^{23},e^{15}+e^{24},e^{16}+e^{34}&(4\, 6)(2\, 5)(1\, 7)&e_1\wedge e_2\\
75421:6&0,0,- e^{12},e^{13},e^{23},e^{15}+e^{24},e^{14}+e^{26}+e^{35}&(3\, 4)(2\, 7)(1\, 6)&e_2\wedge e_3\\
75421:6&0,0,0,- e^{12},e^{14},e^{15}+e^{24},e^{13}+e^{26}+e^{45}&(3\, 4)(2\, 7)(1\, 6)&e_1\wedge e_2\\
74321:12&0,0,0,- e^{12},e^{14}+e^{23},e^{15}+e^{34},e^{16}+e^{35}&(3\, 6)(2\, 5)(1\, 7)&e_1\wedge e_2\\
74321:15&0,0,0,- e^{12},e^{14}+e^{23},e^{15}+e^{34},e^{16}+e^{24}+e^{35}&(3\, 6)(2\, 5)(1\, 7)& e_1\wedge e_2\\
\bottomrule
\end{tabular}
\end{small}
}
\end{table}
}
\begin{remark}\label{remark:infinitelymanydiffeomorphismtypes}
In dimension $7$, there exist continuous families of nilpotent Lie algebras admitting a nice basis. If the parameters are chosen to be rational, the corresponding nilpotent Lie groups admit a compact quotient; the resulting nilmanifolds are pairwise nonisomorphic (see \cite[Theorem 5]{Malcev}). This determines infinitely many diffeomorphism types of Ricci-flat manifolds in any dimension $\geq7$, by taking a product with a torus.
\end{remark}

It is well known that there is only one $6$-dimensional Lie algebra that does not admit any nice basis, namely:
\[(0,0,0,e^{12},e^{14},e^{15} + e^{23} +e^{24}),\]
denoted by $N_{6,1,4}$ in the classification of~\cite{Gong}. However, it can carry Ricci-flat metrics. For example, easy computations show that the following $\sigma$ (written with respect to the basis $\{e_1,\dotsc, e_6\}$) define a Ricci-flat $\sigma$-diagonal metric for any choice of the parameters $g_i$:
\[(1\,3)(2\,6)(4\,5),\qquad(1\,6)(2\,5),\qquad(1\,6)(2\,5)(3\,4).\]
A plane on which the restriction of the curvature is nonzero is given in the first case by $e_2\wedge e_4$, in the other two by $e_1\wedge e_2$.
\begin{corollary}\label{cor:dim6nn}
Every nonabelian 6-dimensional nilpotent Lie algebra has a nonflat Ricci-flat metric.
\end{corollary}

\begin{remark} Given an order two permutation $\sigma$ which is the product of $k$ transpositions, the signature of any $\sigma$-diagonal metric is $(p,q)$ where $p,q\geq k$. Since the coefficients $g_i$ of a $\sigma$-diagonal metric can be chosen arbitrarily all possible signatures with $p,q\geq k$ can be obtained in Theorem \ref{teo:nicedimless7} and Corollary \ref{cor:dim6nn}.
\end{remark}

\FloatBarrier

\bibliographystyle{plain}

\bibliography{diagraminvolutions}

\def\cprime{$'$}
\begin{thebibliography}{10}

\bibitem{AlekseevskiKimelFel}
D.~V. Alekseevsky and B.~N. Kimel{\cprime}fel{\cprime}d.
\newblock Structure of homogeneous {R}iemannian spaces with zero {R}icci
  curvature.
\newblock {\em Funkcional. Anal. i Prilo\v Zen.}, 9(2):5--11, 1975.

\bibitem{AlekseevskyMedoriTomassini:Homogeneous2009}
D.~V. Alekseevsky, C.~Medori, and A.~Tomassini.
\newblock Homogeneous para-{K}\"ahlerian {E}instein manifolds.
\newblock {\em Uspekhi Mat. Nauk}, 64(1(385)):3--50, 2009.

\bibitem{AM03}
A.~{Aubert} and A.~{Medina}.
\newblock {Groupes de Lie pseudo-Riemanniens plats.}
\newblock {\em {Tohoku Math. J. (2)}}, 55(4):487--506, 2003.

\bibitem{Besse}
A.~L. Besse.
\newblock {\em Einstein manifolds}.
\newblock Classics in Mathematics. Springer-Verlag, Berlin, 2008.
\newblock Reprint of the 1987 edition.

\bibitem{BSV15}
N.~{Bokan}, T.~{\v{S}ukilovi\'c}, and S.~{Vukmirovi\'c}.
\newblock {Lorentz geometry of 4-dimensional nilpotent Lie groups.}
\newblock {\em {Geom. Dedicata}}, 177:83--102, 2015.

\bibitem{Boucetta}
M.~Boucetta.
\newblock Ricci flat left invariant {L}orentzian metrics on 2-step nilpotent
  {L}ie groups.
\newblock arXiv:0910.2563v2.

\bibitem{CalvarusoZaeim:lorentzian}
G.~Calvaruso and A.~Zaeim.
\newblock Four-dimensional {L}orentzian {L}ie groups.
\newblock {\em Differential Geom. Appl.}, 31(4):496--509, 2013.

\bibitem{CalvarusoZaeim:neutral}
G.~Calvaruso and A.~Zaeim.
\newblock Neutral metrics on four-dimensional {L}ie groups.
\newblock {\em J. Lie Theory}, 25(4):1023--1044, 2015.

\bibitem{ContiRossi:EinsteinNice}
D.~Conti and F.~A. Rossi.
\newblock Indefinite {E}instein metrics on nice {L}ie groups.
\newblock arXiv:1805.08491.

\bibitem{ContiRossi:Construction}
D.~Conti and F.~A. Rossi.
\newblock {C}onstruction of nice nilpotent {L}ie groups.
\newblock {\em Journal of Algebra}, 525:311 -- 340, 2019.

\bibitem{ContiRossi:RicciFlat}
D.~Conti and F.~A. Rossi.
\newblock {R}icci-flat and {E}instein pseudoriemannian nilmanifolds.
\newblock {\em Complex Manifolds}, 6(1):170--193, 2019.

\bibitem{DaniMainkar:AnosovAutCompactNil}
S.~G. Dani and M.~G. Mainkar.
\newblock Anosov automorphisms on compact nilmanifolds associated with graphs.
\newblock {\em Trans. Amer. Math. Soc.}, 357(6):2235--2251, 2005.

\bibitem{dBO12}
V.~{del Barco} and G.~P. {Ovando}.
\newblock {Free nilpotent Lie algebras admitting ad-invariant metrics.}
\newblock {\em {J. Algebra}}, 366:205--216, 2012.

\bibitem{Derdzinski}
A.~Derdzinski.
\newblock {\em Curvature-homogeneous indefinite {E}instein metrics in dimension
  four: the diagonalizable case}, volume 337 of {\em Contemp. Math.}
\newblock Amer. Math. Soc., Providence, RI, 2003.

\bibitem{DerdzinskiGal}
A.~Derdzinski and {\'S}.~R. Gal.
\newblock Indefinite {E}instein metrics on simple {L}ie groups.
\newblock {\em Indiana Univ. Math. J.}, 63(1):165--212, 2014.

\bibitem{LauretDere:OnRicciNegative}
J.~Deré and J.~Lauret.
\newblock On {R}icci negative solvmanifolds and their nilradicals.
\newblock {\em Math. Nachr.}, 292(7):1462--1481, 2019.

\bibitem{Favre}
G.~Favre and L.~J. Santharoubane.
\newblock Symmetric, invariant, nondegenerate bilinear form on a {L}ie algebra.
\newblock {\em J. Algebra}, 105(2):451--464, 1987.

\bibitem{FinoKath}
A.~Fino and I.~Kath.
\newblock Holonomy groups of {$G_2^*$}-manifolds.
\newblock {\em Trans. Amer. Math. Soc.}, 371(11):7725--7755, 2019.

\bibitem{FinoLujan:TorsionFreeG22}
A.~Fino and I.~Luj\'an.
\newblock Torsion-free {$G^*_{2(2)}$}-structures with full holonomy on
  nilmanifolds.
\newblock {\em Adv. Geom.}, 15(3):381--392, 2015.

\bibitem{Freibert:calibrated}
M.~Freibert.
\newblock {C}alibrated and parallel structures on almost {A}belian {L}ie
  algebras.
\newblock arXiv:1307.2542.

\bibitem{Gong}
M.-P. Gong.
\newblock {\em Classification of nilpotent {L}ie algebras of dimension 7 (over
  algebraically closed fields and {R})}.
\newblock ProQuest LLC, Ann Arbor, MI, 1998.
\newblock Thesis (Ph.D.)--University of Waterloo (Canada).

\bibitem{Guediri2003}
M.~Guediri.
\newblock {Lorentz Geometry of 2-Step Nilpotent Lie Groups}.
\newblock {\em Geom. Dedicata}, 100(1):11--51, 2003.

\bibitem{GuediriBinAsfour}
M.~Guediri and M.~Bin-Asfour.
\newblock Ricci-flat left-invariant {L}orentzian metrics on 2-step nilpotent
  {L}ie groups.
\newblock {\em Arch. Math. (Brno)}, 50(3):171--192, 2014.

\bibitem{Heber:noncompact}
J.~Heber.
\newblock Noncompact homogeneous {E}instein spaces.
\newblock {\em Invent. Math.}, 133(2):279--352, 1998.

\bibitem{IvanovZamkovoy:parahermitian}
S.~Ivanov and S.~Zamkovoy.
\newblock Parahermitian and paraquaternionic manifolds.
\newblock {\em Differential Geom. Appl.}, 23(2):205--234, 2005.

\bibitem{Jensen:scalar}
G.~R. Jensen.
\newblock The scalar curvature of left-invariant {R}iemannian metrics.
\newblock {\em Indiana Univ. Math. J.}, 20:1125--1144, 1970/1971.

\bibitem{Kath:PseudoRiemannian}
I.~Kath.
\newblock Pseudo-{R}iemannian {$T$}-duals of compact {R}iemannian homogeneous
  spaces.
\newblock {\em Transform. Groups}, 5(2):157--179, 2000.

\bibitem{Kath:Indefinite}
I.~Kath.
\newblock Indefinite symmetric spaces with {$\rm G_{2(2)}$}-structure.
\newblock {\em J. Lond. Math. Soc. (2)}, 87(3):853--876, 2013.

\bibitem{Ko1}
B.~Kostant.
\newblock Root systems for {L}evi factors and {B}orel-de {S}iebenthal theory.
\newblock In {\em Symmetry and spaces}, volume 278 of {\em Progr. Math.}, pages
  129--152. Birkh\"{a}user Boston, Inc., Boston, MA, 2010.

\bibitem{Lauret:Einstein_solvmanifolds}
J.~Lauret.
\newblock Einstein solvmanifolds are standard.
\newblock {\em Ann. of Math. (2)}, 172(3):1859--1877, 2010.

\bibitem{LauretWill:Einstein}
J.~Lauret and C.~Will.
\newblock Einstein solvmanifolds: existence and non-existence questions.
\newblock {\em Math. Ann.}, 350(1):199--225, 2011.

\bibitem{LauretWill:diagonalization}
J.~Lauret and C.~Will.
\newblock On the diagonalization of the {R}icci flow on {L}ie groups.
\newblock {\em Proc. Amer. Math. Soc.}, 141(10):3651--3663, 2013.

\bibitem{Magnin}
L.~Magnin.
\newblock Sur les alg\`ebres de {L}ie nilpotentes de dimension {$\leq 7$}.
\newblock {\em J. Geom. Phys.}, 3(1):119--144, 1986.

\bibitem{Malcev}
A.~I. Malcev.
\newblock On a class of homogeneous spaces.
\newblock {\em Amer. Math. Soc. Translation}, 1951(39):33, 1951.

\bibitem{oneill:semiriemannian}
B.~O'Neill.
\newblock {\em Semi-{R}iemannian geometry}, volume 103 of {\em Pure and Applied
  Mathematics}.
\newblock Academic Press, Inc. [Harcourt Brace Jovanovich, Publishers], New
  York, 1983.
\newblock With applications to relativity.

\bibitem{Ov16}
G.~P. Ovando.
\newblock {Lie algebras with ad-invariant metrics: A survey-guide}.
\newblock {\em Rend. Semin. Mat. Univ. Politec. Torino}, 74(1):243--268, 2016.

\bibitem{SchaferSchulteHengesbach}
L.~Sch\"afer and F.~Schulte-Hengesbach.
\newblock Nearly pseudo-{K}\"ahler and nearly para-{K}\"ahler six-manifolds.
\newblock In {\em Handbook of pseudo-{R}iemannian geometry and supersymmetry},
  volume~16 of {\em IRMA Lect. Math. Theor. Phys.}, pages 425--453. Eur. Math.
  Soc., Z\"urich, 2010.

\bibitem{Su16}
T.~{\v{S}ukilovi\'c}.
\newblock {Geometric properties of neutral signature metrics on 4-dimensional
  nilpotent Lie groups.}
\newblock {\em {Rev. Uni\'on Mat. Argent.}}, 57(1):23--47, 2016.

\bibitem{Tamaru}
H.~{Tamaru}.
\newblock {Parabolic subgroups of semisimple Lie groups and Einstein
  solvmanifolds.}
\newblock {\em {Math. Ann.}}, 351(1):51--66, 2011.

\bibitem{WangZiller}
M.~Y. Wang and W.~Ziller.
\newblock Existence and nonexistence of homogeneous {E}instein metrics.
\newblock {\em Invent. Math.}, 84(1):177--194, 1986.

\bibitem{Wolf:TheGeometry}
J.~A. Wolf.
\newblock The geometry and structure of isotropy irreducible homogeneous
  spaces.
\newblock {\em Acta Math.}, 120:59--148, 1968.

\end{thebibliography}

\small\noindent Dipartimento di Matematica e Applicazioni, Universit\`a di Milano Bicocca, via Cozzi 55, 20125 Milano, Italy.\\
\texttt{diego.conti@unimib.it}\\
\texttt{federico.rossi@unimib.it}\medskip

\small\noindent Université Paris-Saclay, CNRS, Laboratoire de mathématiques d’Orsay, 91405, Orsay, France and Universidad Nacional de Rosario, CONICET, Rosario, Argentina.\\
\texttt{viviana.del-barco@math.u-psud.fr}
\end{document}